%% file: arxiv-upload.tex
\documentclass[11pt]{amsart}
\usepackage[pagebackref]{hyperref}
\usepackage{amsmath} 
\usepackage{amsthm}
\usepackage{amssymb}
\usepackage[table, dvipsnames]{xcolor}
\usepackage{tikz-cd}

\usepackage{pgfplots}
\usepackage{tikz}
\usetikzlibrary{decorations.markings}
\usetikzlibrary{decorations.pathreplacing}
\usetikzlibrary{arrows,shapes,positioning}

\newcommand{\nocontentsline}[3]{}
\let\origcontentsline\addcontentsline
\newcommand\stoptoc{\let\addcontentsline\nocontentsline}
\newcommand\resumetoc{\let\addcontentsline\origcontentsline}

\pgfplotscreateplotcyclelist{mycolors}{
	{Lavender},
	{Apricot},
	{JungleGreen}, {Cyan}, {RubineRed}, {RedViolet},{RoyalBlue}, {Brown}, {Gray}, {Salmon}, {Periwinkle}, {Melon}, {MidnightBlue}, {Plum}, {Maroon}, {OrangeRed}, {SeaGreen}, {Tan}, {Dandelion}, {Cyan}, {TealBlue}
}

\pgfplotsset{
	every axis plot/.append style={
		mark=*,
		mark options={scale=0.5}
	}
}

\usepackage{adjustbox}
\usepackage[margin=1.5in]{geometry}
\usepackage{todonotes}
\usepackage{caption}
\usepackage{enumitem}
\usepackage{appendix}
\usepackage{float}
\usepackage{array}
\usepackage{longtable}
\usepackage{tocvsec2}
\usepackage{todonotes}
\usepackage{courier}
\usepackage{listings}
\usepackage{graphicx}

\hypersetup{citecolor=MidnightBlue,
	colorlinks=true,
	linkcolor=MidnightBlue,
	urlcolor=MidnightBlue}

\newtheorem*{rep@theorem}{\rep@title}
\newcommand{\newreptheorem}[2]{%
	\newenvironment{rep#1}[1]{%
		\def\rep@title{#2 \ref{##1}}%
		\begin{rep@theorem}}%
		{\end{rep@theorem}}}
\makeatother

\newtheorem{theorem}{Theorem}[section]
\newtheorem{lemma}[theorem]{Lemma}
\newtheorem{proposition}[theorem]{Proposition}

\newtheorem{algorithm}[theorem]{Algorithm}
\newtheorem{corollary}[theorem]{Corollary}
\newtheorem{theorem*}{Theorem}

\theoremstyle{definition}
\newtheorem{definition}[theorem]{Definition}
\newtheorem{example}[theorem]{Example}
\newtheorem{question}[theorem]{Question}
\newtheorem{remark}[theorem]{Remark}

\newreptheorem{theorem}{Theorem}

\DeclareMathOperator{\C}{\mathcal{C}}

\DeclareMathOperator{\Z}{\mathbb Z}

\title{Egyptian fractions for few primes}

\makeatletter
\def\l@subsection{\@tocline{2}{0pt}{2.5pc}{5pc}{}}
\def\l@subsubsection{\@tocline{2}{0pt}{2.5pc}{5pc}{}}
\makeatother

\usepackage{tocvsec2}

\author[A. Czenky]{Agustina Czenky}
\address{Czenky: Department of Mathematics, University of Southern California,
Los Angeles, CA, USA}
\email{czenky@usc.edu}

\author[E. McGovern]{Emily McGovern}
\address{McGovern: Department of Mathematics, University of Oregon, Eugene, OR, USA}
\email{emilycmcgovern@gmail.com}

\author[J. Plavnik]{Julia Plavnik}
\address{Plavnik: Department of Mathematics, Indiana University, Bloomington, IN, USA. Department of Mathematics and Data Science, Vrije Universiteit Brussel, Brussels, Belgium}
\email{jplavnik@iu.edu}

\author[E. C. Rowell]{Eric Rowell}
\address{Rowell: Department of Mathematics, Texas A\&M University, College Station, TX, USA. School of Mathematics, University of Leeds, Leeds, UK }
\email{rowell@tamu.edu}

\author[A. Watkins]{Abigail Watkins}
\address{Watkins: Department of Mathematics, Indiana University, Bloomington, IN, USA}
\email{abwatk@iu.edu}

\makeatletter
\@namedef{subjclassname@2020}{%
  \textup{2020} Mathematics Subject Classification}
\makeatother



\begin{document}

	\begin{abstract}
	We study solutions to the Egyptian fractions equation with the prime factors of the denominators constrained to lie in a fixed set of primes. We evaluate the effectiveness of the greedy algorithm in establishing bounds on such solutions. Additionally, we present improved algorithms for generating low-rank solutions and solutions restricted to specific prime sets. Computational results obtained using these algorithms are provided, alongside a discussion on their performance.
	\end{abstract}
	
		\maketitle 

  \tableofcontents      
	\section{Introduction}

For a fixed integer $R\geq 1$, the positive integer solutions to the equation \begin{equation}\label{eq:egypt}
    \sum_{i=1}^R\frac{1}{x_i}=1
\end{equation}
have been of interest for millennia (\cite{papyrus}).  A solution $[x_1,\ldots,x_R]$ is called an \emph{Egyptian fraction representation of $1$}. In antiquity, Egyptian fraction representations of more general rational numbers were important as well due to their number system's focus on unit fractions. It was shown by Landau \cite{L} that there are finitely many Egyptian fraction representations of 1, and Takenouchi \cite{takenouchi} provided bounds on the number of solutions, which were improved by Curtiss \cite{curtiss}. More recently Egyptian fractions have been the subject of numerous conjectures in combinatorics and number theory see \cite{Guy}).  
One early application  used the fact that there are finitely many solutions to prove that there are finitely many finite groups with precisely $R$ conjugacy classes (which is equal to the number of irreducible complex representations as well).  This, in turn, was adapted to prove that there are finitely many integral fusion categories with exactly $n$ equivalence classes of simple objects \cite[Proposition 8.3]{ENO}.  This was a precursor to the rank-finiteness theorem for modular categories \cite{BNRW}: there are finitely many modular categories of any rank $R$, i.e., with exactly $R$ isomorphism classes of simple objects. This used a non-integral generalization of \eqref{eq:egypt} called the $S$-unit equation.

To classify integral fusion categories of low rank one needs to analyze the explicit solutions to \eqref{eq:egypt}. Unfortunately the general bound on the $x_i$ are doubly exponential in $R$, making the problem of finding the solutions for even modest values of $R$ intractable. With further assumptions, such as modularity, one has some powerful divisibility conditions which simplifies the problem, extending the values of $R$ for which a classification is feasible.  This was undertaken in the special case of \emph{modular} tensor categories (MTCs) of \emph{odd} dimension in \cite{BR}, and continued in \cite{CP}.   One of the results in \cite{BNRW} implies that the prime divisors of $\dim(\mathcal{C})$ coincide with the prime divisors of the so-called $T$-matrix of $\mathcal{C}$, which can be bounded in terms of the rank $R$.  In particular, we find that in the modular case we do not need all solutions, rather we only need those with a certain fixed set of primes.

This inspires a purely number theoretical problem: given a finite set of primes $\mathcal{S}$, find all solutions $[x_1,\ldots,x_R]$ such that the prime divisors of each $x_i$ lie in $\mathcal{S}$.  In principal this should be exponentially easier than the general solutions, since the bound for $k$ is a logarithm of the bound for $p^k$. Other restrictions to Egyptian fractions have been studied in the past, see for example \cite{MartinShi2025}.

In this paper we explore this possibility, with a view towards understanding the bounds.  For the general problem the bounds are obtained in \cite{takenouchi,curtiss}.  The idea there is to construct a ``greedy" solution with the largest possible values as follows: for fixed $R>1$, choose an integer $a_1\geq 1$ so that $\frac{1}{a_1}$ is as large as possible with $\frac{1}{a_1}<1$.  Continuing, for $i<R$ one chooses $a_{i}\in\mathbb{N}$  so that $\frac{1}{a_i}$ is as large as possible with  $\frac{1}{a_i}<1-\sum_{j=1}^{i-1}\frac{1}{a_j}$. Then one defines $a_R$ by $\frac{1}{a_R}=1-\sum_{j=1}^{R-1}\frac{1}{a_j}$.  Of course \emph{a priori} $a_R$ could fail to be an integer, but it is shown that it is an integer.  In fact, this algorithm is related to Sylvester's sequence \cite{Syl} defined inductively as follows: $s_1=2$, $s_i=s_{i-1}(s_{i-1}-1)+1$.  This sequence starts: $2,3,7,43,1807,\ldots$, and yields a solution to \eqref{eq:egypt}  as $[s_1,s_2,\ldots,s_{R-1},s_R-1]$.  It turns out, that \emph{any} rank $R$ solution $[x_1,\ldots,x_R]$ will have $x_i\leq s_i$.

This inspires a similar approach for sharpening the bound on solutions $[x_1,\ldots,x_R]$ such that there is a finite set $\mathcal{S}$ such that if a prime $p\mid x_i$ then $p\in \mathcal{S}$: as we choose our integer $a_i$ at each step above, we simply ensure that it lies in the set of integers whose prime factors are in $\mathcal{S}$.  Since we have a bound on the $x_i$ provided by Sylvester's sequence, this is a finite problem.  Of course it could be that at the last step we find $\frac{1}{a_R}$ does not have $a_R\in\Z$.

In this paper we explore this idea experimentally.  There are a number of approaches.  The first is to simply generate all solutions for some small $R$ to gain some insight.  The second is to consider a small number of primes, (say 1 or 2) and generate solutions algorithmically.  We find that the greedy approach fails.  However, with the data generated we can do statistical analyses to try to refine our approach, and discover some conjectures, some of which we prove, while others we leave as questions.

\stoptoc

\section*{Acknowledgements}
We are grateful to Nadav Kohen for helpful discussions on integer sequences and computer programming. This project began while the authors were participating in a workshop at the \emph{American Institute of Mathematics}, whose hospitality and support are gratefully acknowledged. A.C. was partially supported by Simons Collaboration Grant No. 999367.
E.M. was partially supported by NSF grant DMS-2039316. 
J.P. was partially supported by US NSF Grant DMS-2146392 and by Simons Foundation Award \#889000 as part of the Simons Collaboration on Global Categorical Symmetries. 
E.C.R. was partially supported by US NSF Grant DMS-2205962 and a Royal Society Wolfson Visiting Fellowship. 

\resumetoc

\section{Solutions for a fixed set of primes}
Fix a set of primes $p_1< \dots < p_m.$	In this section, we give some results regarding the existence of solutions to the Egyptian fraction equation involving only the primes $p_1, \dots, p_m$.

\begin{definition} Let $R>0$.
We say that a tuple $[x_1, \dots, x_R]_{p_1, \dots, p_m}$  is a \emph{solution in rank $R$} if 
\begin{align}\label{eq: mult primes}
	\sum\limits_{i=1}^R \frac{1}{x_i}=1,
\end{align}
and
\begin{itemize}
    \item each $x_i=\Pi_{j=1}^m p_j^{a_{i,j}}$ for some $a_{i,j}\geq 0$,
    \item $x_i\leq x_j$ for $i\leq j,$
    \item for all $1\leq j\leq m$ there exists some $1\leq i\leq R$ such that $a_{i,j}>0$.
\end{itemize}

We call $R$ the \emph{rank} of the given solution. 
\end{definition}

We include the third condition in the above definition as we are interested in finding solutions where each of the primes $p_1, \dots, p_m$ shows up as a factor in at least one of the denominators. For example, $[3,3,3]_3$ is valid in our notation, but $[3,3,3]_{3,5}$ is not.

\begin{remark}\label{remk: bound on R}
To have solutions at rank $R$, we must have that at least one of the primes $p_1, \dots, p_m$ is smaller than $R$.
\end{remark}

Our first result states a necessary condition for the existence of a solution in a given rank $R$.

\begin{proposition}\label{Prop: condition for rank}
Let $n\geq 1$ such that $p_j\equiv 1 \mod n$ for all $1\leq j\leq m$, and suppose that there exists a solution in rank $R$. Then $R\equiv 1 \mod n$. 
\end{proposition}

\begin{proof}
 Let $[x_1, \dots, x_R]_{p_1, \dots, p_m}$ be a solution in rank $R$. Then
 \begin{align*}
   1= \sum\limits_{i=1}^R \frac{1}{x_i}= \sum\limits_{i=1}^R \frac{\Pi_{j=1}^mp_j^{a_{R,j}-a_{i,j}}}{\Pi_{j=1}^mp_j^{a_{R,j}}},
 \end{align*}
 so multiplying on both sides by $\Pi_{j=1}^mp_j^{a_{R,j}}$ we obtain 
 \begin{align*}
   \Pi_{j=1}^mp_j^{a_{R,j}}= \sum\limits_{i=1}^R \left(\Pi_{j=1}^mp_j^{a_{R,j}-a_{i,j}}\right).
 \end{align*}
 Taking congruence mod $n$ on both sides, it follows that 
 \begin{align*}
     1\equiv \sum\limits_{i=1}^R 1 \equiv R \mod n,
 \end{align*}
 as desired. 
\end{proof}

\begin{definition}
    Given a set of fixed primes, we call an  integer $R$ satisfying the proposition above an \emph{admissible} rank. 
\end{definition}

\begin{corollary}\label{cor: gcd condition}
	If there exists a solution in rank $R$, then $\gcd(p_1-1, \dots, p_m-1)$ divides $R-1$.
\end{corollary}
 \begin{proof}
 	Take $n=\gcd(p_1-1, \dots, p_m-1)$ in Proposition \ref{Prop: condition for rank}.
 \end{proof}

 \begin{remark}
The converse is not true in general. That is, we can have that $\gcd(p_1-1, \dots, p_m-1)$ divides $R-1$ and yet no solutions exist at rank $R$.
For example, there are no solutions of the form $[x_1, x_2, x_3, x_4, x_5]_{2,11}$, see Appendix \ref{App: solutions: rank 5}.
\end{remark}

\begin{corollary}
    If $p_1>2$, there are no solutions at even rank. 
\end{corollary}
\begin{proof}
    Since $p_i-1$ is even for all $i$, 2 divides the $\gcd(p_1-1, \dots, p_m-1)$ and thus divides $R-1$.
\end{proof}

In other words, all solutions in even rank must have at least one $x_i$ even (that is, $p_1 = 2$). 

\begin{remark}
    This result ties in nicely with a well-known-fact about MTCs, which states that integral MTCs are odd-dimensional if and only if they have odd rank, see for example \cite{BR, CP}.
\end{remark}

\subsection{Solutions for only one prime}
Fix a prime $p\geq 2.$ The focus of this section is to study the behavior of solutions to the Egyptian fraction equation that involve only powers of $p$, and to discuss some of the related results already in the literature.

Recall that, as per our notation, a tuple $[x_1, \dots, x_R]_p$  is a solution to the equation
\begin{align}\label{eq: one prime}
	\sum\limits_{i=1}^R \frac{1}{x_i}=1.
\end{align}
such that every $x_i$ is a power of $p$.
Then Proposition \ref{Prop: condition for rank} gives a necessary condition for the existence of a solution in a given rank $R$, which we state below.

\begin{corollary}
	If there exists a solution in rank $R$, then $p-1$ divides $R-1$.
\end{corollary}

\begin{proof}
	Take $m=p-1$ in Proposition \ref{Prop: condition for rank}.
\end{proof}

\begin{remark}
	The above proposition tells us that for $p>2$  there can only be solutions in odd rank, and gives no restrictions for $p=2$.
\end{remark}

\begin{theorem}
Let $R$ be an admissible rank for $p$. Then there exists a solution $[x_1, \dots, x_R]_p$. In particular, the following is one such solution:
\begin{align*}
		\left[  \underbrace{p, \dots, p}_{p-1}, \underbrace{p^2, \dots, p^2}_{p-1}, \dots, \underbrace{p^{\frac{R-1}{p-1}}, \dots, p^{\frac{R-1}{p-1}}}_{p-1}, p^{\frac{R-1}{p-1}} \right]_p .
	\end{align*}
\end{theorem}
\begin{proof}
 First, note that for a fixed $k$ we have
 $$\sum_{i = 1}^{p-1} \frac{1}{p^{k}} = \frac{p-1}{p^{k}}$$
 and so the sum of the fractions whose denominators are the first $R - 1$ terms of the conjectured solution is $$\sum_{k = 1}^{\frac{R-1}{p-1}} \frac{p-1}{p^{k}}.$$
 From now, take $N = \frac{R-1}{p-1}$. We can simplify this sum in the following way
 $$ \sum_{k = 1}^{N} \frac{p-1}{p^k} = (p-1)\sum_{k = 1}^{N}\frac{1}{p^k} = (p-1) \frac{\sum_{k = 1}^{N} p^{N-k}}{p^{N}}.$$

 See that $(p-1)(1 + p + ... + p^{N-1}) = p - 1 +p^{2} - p + ... + p^{N} - p^{N-1} = p^{N} - 1$. It follows that the sum of the (inverses of the)  first $R-1$ terms of the sequence is $ \frac{p^{\frac{R-1}{p-1}} - 1}{p^{\frac{R-1}{p-1}}},$ which means that the proposed sequence is a solution in rank $R$, as desired. 
\end{proof}

We give below a bound for the values of the denominators in a given solution for fixed rank. We refer the reader to Section \ref{sec: greedy alg} for a more in-depth discussion of such bounds in the case of more than one prime. 

\begin{proposition}\label{prop: bound for one rank}
	Let $R\geq 1$ such that $p-1$ divides $R-1$, and let $[x_1, \dots, x_R]_p$ be a solution in rank $R$. Then $x_i\leq p^{\frac{R-1}{p-1}}$ for all $i=1, \dots, R$. 
\end{proposition}

\begin{proof}
This follows from the fact that 
\begin{align*}
		\underbrace{p, \dots, p}_{p-1}, \underbrace{p^2, \dots, p^2}_{p-1}, \dots, \underbrace{p^{\frac{R-1}{p-1}}, \dots, p^{\frac{R-1}{p-1}}}_{p-1}
\end{align*}
is the sequence of $R-1$ terms whose sum is maximal, which is easy to see. Additionally, we see that adding $\frac{1}{p^{\frac{R-1}{p-1}}}$ gives us 1, and so for $k > \frac{R-1}{p-1}$, adding $\frac{1}{p^k}$ will result in a total sum less than 1. Since the sequence we are looking at is maximal through $R-1$ terms, there is no way to get a sequence whose reciprocals sum to one with a larger denominator than $p^{\frac{R-1}{p-1}}$.
\end{proof}

\begin{remark}
	In the context of MTCs, this means that if $\C$ is an MTC of rank $R$ such that every simple object has dimension a power of $p$, then $\dim(\C)\leq p^{\frac{R-1}{p-1}}$.
\end{remark}

In Appendix \ref{App: solutions: one prime}, we implement our code from Appendix \ref{appendix: one prime code} to list all solutions involving one prime for small admissible ranks. In particular, we colored the solutions where the highest possible power of $p$ makes an appearance.

\subsubsection{Counting  solutions}\label{sec:counting-sols}

 In this section we discuss briefly the following question.
\begin{question}
	Can we know the number $\mathfrak{s}_p(R)$ of solutions for a given prime $p$ and rank $R$?
\end{question}

This counting problem appears in the literature related to several others, such as counting “proper words” \cite{EL} and
 counting nonequivalent complete rooted $t$-ary trees \cite{FP, EHP, HKW}. 
In particular, this question is studied in \cite{EHP} in relation to Huffman codes and rooted trees, where they give a detailed survey on the existing literature related to the problem at the time and prove an asymptotic result; they remark that a trivial upper bound is $\mathfrak{s}_p(R)\leq 2^{\frac{R-1}{p-1}}$. They also list the number of solutions (at low ranks) for the primes 3, 5 and 7 in \cite[Table 1]{EHP}. 

Moreover, in \cite{EHK} they study the complexity of computing the number $\mathfrak{s}_p(R)$, proving that it is of the order $O(N^{1+\epsilon})$ (for any $\epsilon > 0$); in contrast, the fastest previously published algorithm  \cite{EL} which appeared in 1972 had a
complexity of $O(N^3)$ operations.

The number of solutions in rank $R$ is the number of partitions of $1$ into $R$ powers of $\frac{1}{p}$. For $p=2$, this is listed as sequence \href{https://oeis.org/A002572}{A002572} in \cite{oeis2025}. 
We include below a table showing the number of solutions for small rank, obtained from the data generated in Appendix \ref{App: solutions: one prime}. We note that this matches \cite[Table 1]{EHP} for $p=3, 5$ and $7$. Note that entries of N/C correspond to solutions that exist but were not able to be computed by the authors. 
\begin{table}[H]
    \centering
    \begin{tabular}{|c|c|c|c|c|c|c|c|}
    \hline
    & $t=1$ & $t = 2$ & $t = 3$ & $t = 4$ & $t = 5$ & $t = 6$ & $t = 7$  \\ \hline
         $p = 3$ & 1 & 1 & 2 & 4& 7& 13 & 25 \\ \hline
         $p = 5$ &  1 & 1 & 2 & 4& 8& 16 & 31 \\ \hline
         
         $p = 7$ &  1 & 1 & 2 & 4& 8& 16 & 32\\ \hline

          $p = 11$ &  1 & 1 & 2 & 4& 8& 16 & \cellcolor{JungleGreen!50} N/C\\ \hline
             $p = 13$ &  1& 1 & 2 & 4& 8& 16 & \cellcolor{JungleGreen!50} N/C\\ \hline
              $p = 17$ &  1 & 1 & 2 & 4& 8& 16 & \cellcolor{JungleGreen!50} N/C\\ 
             \hline
    \end{tabular}
    \caption{Number of solutions at rank $R=(p-1)t +1$}
    \label{tab:number of solutions one prime}
\end{table}
For $p=3$, the sequence is
$ [1, 1, 2, 4, 7, 13, 25, 48, 92, 176, \dots]$ 
	which is \href{https://oeis.org/A176485}{A176485} in \cite{oeis2025}. 
	For $p=5$, 
		 $[1, 1, 2, 4, 8, 16, 31, 61,\dots]$
	which is \href{https://oeis.org/A194628}{A194628}. Similarly,  for $p = 7$ we get \href{https://oeis.org/A194630}{A194630}.

\section{Algorithms}

We give here the algorithms we used to compute low rank solutions and all the data found in this paper. The explicit codes based on each algorithm can be found in the appendices. 

\subsection{A Greedy Algorithm for Finding a Maximum Denominator}\label{sec: greedy alg}

One of our goals for this project was to discover if a greedy-type algorithm could accurately provide a bound for the maximum denominator found in an Egyptian fraction problem for specific primes. We present the following algorithm which we considered. 
\begin{algorithm}
\label{alg: greedy}
    \normalfont 
    Input: $p$ and $q$, two prime numbers, and $R$, the rank at which we wish to find the greedy bound.
    \newline
    Output: $\Gamma$, the proposed bound for a denominator appearing in an Egyptian fraction representation with primes $p,q$ in rank $R$.
    \begin{enumerate}
        \item Initialize some variables: $\sigma$ (initialized to 0) will keep a running sum and $\phi$ (initialized to an empty list) will keep track of the denominators we have used. 
        \item Iterate over all $p^{i}q^{j}$ for $i,j > 0$ in increasing order:
        \begin{enumerate}
            \item if $\sigma + \frac{1}{p^iq^j} < 1$ and the length of $\phi < R$, add $p^iq^j$ to $\phi$ and $\frac{1}{p^iq^j}$ to $\sigma$. 
        \end{enumerate}
        \item Stop iteration when there are $R-1$ items in $\phi$. 
        \item Calculate $\gamma_ = \frac{1}{1 - \sigma}$ and find $\Gamma$, which we define as the smallest $p^iq^j$ greater than or equal to $\gamma$. 
        \item Return $\Gamma$.
    \end{enumerate}
\end{algorithm}
At each step this algorithm considers the running sum  $\sigma$ of the denominators we have chosen. The reciprocal of  $1 - \sigma$ will give us the smallest denominator we can choose in order for our sum to remain less than 1. If the greedy algorithm gives us an exact solution to the Egyptian fraction problem, after $R-1$ iterations we will have $\sigma = \frac{n-1}{n}$ and an obvious choice for final denominator is this $n$. If not, we can use the same logic (choosing the smallest possible denominator bigger than $\frac{1}{1-\sigma}$) to find the final denominator and our greedy algorithm bound. 
\begin{example}
   Consider the case when $p = 2$, $q = 7$ and $R = 7$. If we follow the greedy algorithm through the first 6 iterations, we get
   $$\frac{1}{2} + \frac{1}{4} + \frac{1}{7} + \frac{1}{14} + \frac{1}{32} + \frac{1}{256} = \frac{1791}{1792}.$$

   The final step of this algorithm returns 1792 and adding $\frac{1}{1792}$ to the sum gives us 1 and therefore this list of denominators is a solution to the Egyptian fraction problem.  
\end{example}
\begin{example}
  Consider the case when $p = 2$, $q = 5$ and $R = 7$. If we follow our greedy algorithm through the first 6 steps, we will choose 2, 4, 5, 25, 125 and 512. However, the sum for this list is
  $$\frac{1}{2} + \frac{1}{4} + \frac{1}{5} + \frac{1}{25} + \frac{1}{125} + \frac{1}{512} = 
  \frac{63997}{64000}.$$

 The left over after 6 choices $\frac{3}{64000}$ is about $\frac{1}{21333}$. So our algorithm will take  take the next largest $2^i5^j$, in this case 25000, as the greedy bound. However, note that $\frac{63997}{64000} + \frac{1}{25000} = \frac{1599989}{1600000}$ and so this bound of 250000 is not tight (i.e. it does not occur in a solution).  The maximum denominator that appears in a solution for this $p$, $q$ and $R$ is $2500$. 
  
\end{example}
We list below the bounds for the denominators in an Egyptian fractions solution according to our greedy algorithm. Entries in bold font are those in which a solution with the greedy bound as the maximum denominator occurring exists. Squares colored pink are ranks at which a solution does not exist. The greedy algorithm will still run and give a solution, but as a solution does not exist with at least one denominator that is $p^iq^j$ for $i,j > 0$, its output will be meaningless. The one red entry is a rank at which there exists a solution with a denominator greater than the greedy algorithm bound. We will discuss this failure of the greedy algorithm in the following section. 
\begin{table}[h]
    \centering
    \begin{tabular}{|c|c|c|c|c|}
        \hline
        $q$ & $R$  = 5 & $R$ = 6 & $R$ = 7 & $R$ = 8 \\ \hline
        3 & \textbf{216} & \textbf{1944} & 39366 & 1417176 \\ \hline
        5 & \textbf{100} & \textbf{500} & 25000 & 156250 \\ \hline
        7 & \textbf{28} & \textbf{224} & \textbf{1792} & \textbf{14336} \\ \hline
        11 &\cellcolor{WildStrawberry!50} & \textbf{352} & 1331 & \textbf{42592} \\ \hline
        13 & \cellcolor{WildStrawberry!50}& \textbf{104} & \textcolor{red}{676} & 3328 \\ \hline
        17 & \cellcolor{WildStrawberry!50}& \textbf{272} & \textbf{4624} & \textbf{78608} \\ \hline 
        19 & \cellcolor{WildStrawberry!50}& \cellcolor{WildStrawberry!50}& 512 & 9728 \\ \hline
    \end{tabular}
    \caption{Greedy algorithm for small rank and $q$ when $p = 2$}
    \label{tab:my_label}
\end{table}
\subsubsection{The Failure of the Greedy Algorithm}\label{sec: failure of greedy}
For $p = 2$, the greedy algorithm gives a correct bound for all primes $q$ in ranks less than or equal to 6. In rank 7, we find our first counterexample, when $q = 23$. Here, Algorithm \ref{alg: greedy} tell us that the maximum denominator that should appear is 676. However, [2, 4, 8, 13, 32, 64, 832] is an Egyptian fraction solution, which contradicts our conjecture that we can use a greedy algorithm to put a bound on the maximum denominator that appears.

Another interesting case to consider is $p = 2, q = 23$ in rank 8. Here the greedy algorithm gives us 1472 as a bound, and this corresponds to the solution [2, 4, 8, 16, 23, 64, 368, 1472]. However, we also have a solution [2, 4, 8, 16, 23, 92, 128, 2944]. This example is also of interest because the greedy algorithm gives a bound with which we can find a corresponding solution (unlike when $q = 13$, $R = 7$). It is mysterious to us what causes the failure of the greedy algorithm at these ranks.  We found several other counter examples in small rank that we present below.

\begin{table}[H]
    \centering
    \begin{tabular}{|c|c|c|c|}
        \hline
        \textbf{Rank} & $\boldsymbol{q}$ & \textbf{Greedy Bound} & \textbf{Maximum Denominator}\\ \hline
        8 & 23 & 1472 & 2944\\ \hline
        9 & 13  & 43624 & 140608\\ \hline
        9 & 23  & 5888& 11776\\ \hline
        10 & 5  & 15625000 & 51200000\\ \hline
        10 & 13  & 262144 & 346112\\ \hline
        10 & 23 & 47104 & 94208\\ \hline
        10 & 29  & 24389 & 29696\\ \hline
        10 & 59 & 1888 & 7552 \\ \hline
        10 & 107 & 3424 & 27392 \\ \hline
    \end{tabular}
    \caption{Greedy algorithm failure in small ranks for $p = 2$}
    \label{tab:greedy algorith failure}
\end{table}

\subsection{Recursive Algorithm for Generating Solutions With Two Primes}
\label{subsection: recursive algorithm}
We present  an alternative algorithm to that found in \cite{BR} to generate solutions for fixed $p$, $q$ at any rank. Implementations of this algorithm in Sage and Julia can be found in the accompanying \href{https://github.com/aswatkin/Remarks-on-Egyptian-Fractions-For-Few-Primes.git}{GitHub repository}  to this paper. 

\begin{algorithm}
    \normalfont Input: $p$ and $q$, two prime numbers, $R$, the rank we wish to find the greedy bound at, $\phi$, a solution being constructed, and $\Lambda$, a list of all possible solutions. 

    \begin{enumerate}
        \item Iterate over the items in $\phi$ and sum their reciprocals - call this sum $\sigma$.
        \item If $\sigma > 1$, terminate the process. 
        \item If $\sigma < 1$ and there are $R$ items in $\phi$, terminate. 
        \item If $\sigma = 1$ and there are $R$ items in $\phi$, add $\phi$ to $\Lambda$ and terminate. 
        \item Define $\frac{m}{n} = 1 - \sigma$ and $s$ the number of items in $\phi$. 
        \item Iterate over all $p^{i}q^{j}$ such that $max(\phi) \leq p^{i}q^{j} \leq \frac{n(R-s)}{m}$
        \begin{enumerate}
            \item Add $p^iq^j$ to $\phi$.
            \item Recursively call the algorithm with the updated $\phi$.
            \item Remove $p^iq^j$ from $\phi$. 
        \end{enumerate}
    \end{enumerate}
    This process will terminate with $\Lambda$ containing all possible solutions to the Egyptian fraction problem with the primes $p,q$ in rank $R$. Besides the recursive nature of the algorithm, the main difference from the one presented in \cite{BR} is the upper bound in the loop in step 6. If we have $\frac{m}{n}$ left to account for, and $R-s$ denominators to choose, the average of $\frac{1}{d}$ for the denominators we have left to pick is $\frac{m}{n(R-s)}$. Inverting this gives us that the average of the denominators must be $d = \frac{n(R-s)}{m}$. We must choose at least one denominator less than this average $d$ (so $\frac{1}{d}$ is greater than average), hence we can exclude any $d > \frac{n(R-s)}{m}$ as our next choice. 
\end{algorithm}
\begin{remark}
    A shortcoming of the algorithm as written is that it will record, for a given $p,q,R$ triple, solutions that only use one of the two primes (for example, in rank $4$, $[2,4,8,8]$ appears as a solution for $p = 2$ and $q$ any other prime).
    The number of solutions with a single prime divisor is well known (see section \ref{sec:counting-sols} and appendix \ref{appendix: one prime code}). As such, we can ensure that for a given $p,q,R$ triple our algorithm is producing new solutions by simply checking that the number of solutions we generate is greater than $\mathfrak{s}_p(R)+\mathfrak{s}_q(R)$.   
\end{remark}
\begin{remark}
    In the implementation of the algorithm, we generate a list of possible denominators. Because we want this list to be finite, we must choose some highest power $k$ so that $p^iq^j$ is included as a possible denominator only when $i,j \leq k$. There are several ways to choose this $k$. The ideal starting point would be $\max(\log_p(S_R), \log_q(S_R))$, where $S_R$ is the $R^{th}$ term of the Sylvester sequence which gave our previously known bound for the maximum denominator found in this sequence. However, even in small ranks, this number is much bigger than necessary. In rank 7, we have $S_7 = 10650056950807$ and so our $k$ would be 43, however, for $p=2$ and $q$ any prime, the highest power of either prime that occurs in a solution is 8. So it is more beneficial to experiment with different $k$ values in order to find a value big enough to not affect the maximum denominator and small enough for the code to execute efficiently. 
\end{remark}
\begin{remark}
    It is reasonable to ask about the run time of this algorithm. It is clear that step (1) is $O(R)$ and steps 2-5 run in constant time. If $p < q$ (we can assume this without loss of generality), the loop in step (6) will iterate at most $\log_p(\frac{n(R-s)}{m})$ times. It is hard to put a bound on  $\frac{n(R-s)}{m}$, as it is increasing with each denominator added to a potential solution, so we can't say much more about the run time of this algorithm. 
\end{remark}

\subsection{Algorithm for low rank solutions}
For small $n$, it is possible to construct all solutions.  The bottleneck is the upper bound on the largest $x_i$: it is doubly exponential in $n$.
\begin{algorithm}\label{alg:low rank} Construct a set of lists $A_n$ so that $X\in A_n$ satisfies \begin{enumerate}
\item $\sum_{i=1}^n\frac{1}{X[i]}=1$ and 
 \item $X[j]\leq X[j+1]$ for all $1\leq j\leq n-1$.

\end{enumerate}
This is accomplished as follows:
 \begin{enumerate}
 \item Input: An integer $n$
 \item Define auxilary sequence $u_1=1$, $u_i=u_{i-1}(u_{i-1}+1)$ for $2\leq i\leq n$.
 \item Initialize $A_1=\{[2],\ldots,[n]\}$
\item For $1<j<n$ define $A_j$ inductively as follows: \\for $X\in A_{j-1}$ and $j+1\leq k \leq u_j(n-j+1)$ if $\frac{1}{k}+\sum_{i=1}^{j-1}\frac{1}{X[i]}<1$ and $k\geq X[j-1]$ include $[X[1],\ldots X[j-1],k]$ in $A_j$.
\item Given $A_{n-1}$ construct $A_n$ as follows: for $X\in A_{n-1}$ define $f_X:=1-\sum_{i=1}^{n-1}\frac{1}{X[i]}$. If $\frac{1}{f_X}\in\Z$ and $f_X\geq X[n-1]$ then include $[X[1],\ldots,X[n-1],f_X]$ in $A_n$.
 \item Output the list $A_n$.
 \end{enumerate}
\end{algorithm}
\begin{remark} 
 The lower bounds in step (4) are justified as follows. Suppose that $[X_1,\ldots,X_n]$ is a weakly increasing solution. First note that if some $X_i<i$ then $X_1\leq\cdots X_i<i$, so that $\sum_{j=1}^i\frac{1}{X_j}>i(1/i)=1$. So we have $X_i\geq i$ for each $i$. If some $X_j=j$ then we must have $X_i\leq j$ for each $i\leq j$. Thus $\sum_{i=1}^j\frac{1}{X_i}\geq \frac{j}{j}=1 $. This implies that this is only possible for $j=n$ and we have the trivial solution. Note that the trivial solution is produced by the algorithm since the last step does not use the bound.
 \end{remark}
 \begin{remark}
The upper bound is justified in \cite{curtiss,takenouchi}, see also \cite{BR}. Since it is doubly exponential the algorithm is quite inefficient. In fact, to implement this code for $n=6$ we found it necessary to construct $A_5$ in more manageable pieces. 
 \end{remark}

\section{Solutions with the prime $2$}

We look here at solutions that include the prime 2. That is, solutions for a fixed set of primes $\{p_1=2, p_2, \dots, p_m\}.$ We recall that when $p_1=2$ Corollary \ref{cor: gcd condition} gives no restrictions, as $\gcd(1, p_2-1,\dots,p_m-1)=1.$

\begin{lemma}
If $[x_1, \dots, x_R]_{2, p_2, \dots, p_m}$ is a solution in rank R, then $[x_1, \dots, x_{R-1}, 2x_R, 2x_R]_{2, p_2,  \dots, p_m}$ is a solution in rank $R+1.$
\end{lemma}

\begin{proof}
    This follows from a direct computation. 
\end{proof}

As a consequence of the lemma above, if we are able to find a solution $[x_1, \dots, x_R]_{p_1, \dots, p_m}$ at rank $R$, then we are guaranteed the existence of solutions at all higher ranks as well. Hence for a fixed set of primes $\{p_1=2, p_2, \dots, p_m\}$ it makes sense to wonder what the smallest rank at which there exists a solution in.

\begin{example}
   For the primes $\{2,3\}$, we can find solutions at any rank $R\geq 3$. In fact, $[2,3,6]_{2,3}$ is a solution at rank $3$. 
\end{example}

\subsection{The case $\{2, p\}$} We present here a table showing the lowest rank at which a solution with the primes $\{2,p\}$ exists. This data was generated using the code given in Appendix \ref{sec:code-two-primes}.

\begin{table}[H]
    \centering
    \begin{tabular}{|c|c|c|c|c|c|}
        \hline
        $\boldsymbol{p = }$  & \textbf{Rank} & $\boldsymbol{p = }$ & \textbf{Rank} & $\boldsymbol{p = }$ & \textbf{Rank}  \\ \hline
        \textbf{3} & 3 & \textbf{5} & 4 & \textbf{7} & 5 \\ \hline
        \textbf{11} & 6 & \textbf{13} & 6 & \textbf{17} & 6 \\ \hline
        \textbf{19} & 7&  \textbf{23} & 8 & \textbf{29} & 8 \\ \hline
        \textbf{31}&  9& \textbf{37} & 8 & \textbf{41} & 8 \\ \hline
        \textbf{43} & 8 & \textbf{47} & 10 & \textbf{53} & 9 \\ \hline
        \textbf{59} & 10& \textbf{61} & 10 & \textbf{67} & 9 \\ \hline
        \textbf{71} & 10 & \textbf{73} & 9 & \textbf{79} & 11 \\ \hline
        \textbf{83} & 10 & \textbf{89} & 10 & \textbf{97} & 9 \\ \hline
        \textbf{101} & 10 & \textbf{103} & 11 & \textbf{107} & 10 \\ \hline 
        
        \textbf{109} & 11 & \textbf{113} & 10 & \textbf{127} & 13 \\ \hline \textbf{131} & 10 &
        \textbf{137} & 10 & \textbf{139} & 11 \\ \hline \textbf{149} & 11 & \textbf{151} & 11 & \textbf{157} & 12 \\ \hline 
        \textbf{163} & 11 & \textbf{167} & 12 & \textbf{173} & 12 \\ \hline \textbf{179} & 12 &
        \textbf{181} & 12 & \textbf{191} & 14 \\ \hline \textbf{193} & 10 &
        \textbf{197} & 11 & \textbf{199} & 12 \\ \hline \textbf{211} & 12 &
        \textbf{223} & 13 & \textbf{227} & 12 \\ \hline \textbf{229} & 12 &
        \textbf{233} & 12 & \textbf{239} & 12 \\ \hline \textbf{241} & 12 &
        \textbf{251} & 14 & \textbf{257} & 10 \\\hline        
        
    \end{tabular}
    \caption{Lowest rank at which a solution with $2$ and given $p$ appears}
    \label{tab: lowest rank label}
\end{table}

Looking at the data, the rank seems to increase slowly as $p$ increases. We note also that the rank sometimes remains constant for consecutive primes, and ranks for consecutive primes tend to differ by at most 1, which suggests a logarithmic growth. 

A least-squares fit to the available data suggests that the rank grows approximately logarithmically with $p$. Specifically, the best fit for the data is given by the model:
\[
\text{Rank}(p) \approx 2.07 \cdot \ln(p) + 0.88.
\]
We also observe that $\ln(p)$ seems to give a bound for the lowest rank $R$ in terms of $p$.
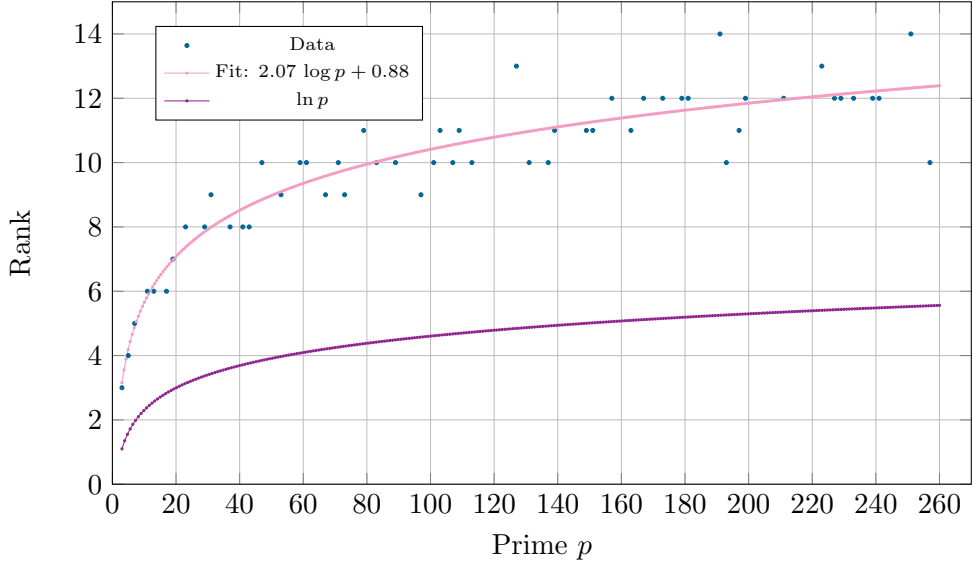
\begin{figure}[H]
	\centering
	\begin{tikzpicture}
		\begin{axis}[
			width=13cm,
			height=8cm,
			xlabel={Prime $p$},
			ylabel={Rank},
			grid=both,
			legend style={at={(0.05,0.95)}, anchor=north west},
			xmin=0, xmax=270,
			ymin=0, ymax=15
			]
			
			\addplot[
			only marks,
			mark=*,
			MidnightBlue,
			mark size=1.5pt
			] table {
				x  y
				3   3
				5   4
				7   5
				11  6
				13  6
				17  6
				19  7
				23  8
				29  8
				31  9
				37  8
				41  8
				43  8
				47 10
				53  9
				59 10
				61 10
				67  9
				71 10
				73  9
				79 11
				83 10
				89 10
				97  9
				101 10
				103 11
				107 10
				109 11
				113 10
				127 13
				131 10
				137 10
				139 11
				149 11
				151 11
				157 12
				163 11
				167 12
				173 12
				179 12
				181 12
				191 14
				193 10
				197 11
				199 12
				211 12
				223 13
				227 12
				229 12
				233 12
				239 12
				241 12
				251 14
				257 10
			};
			
			\addplot[
			Lavender,
            mark options={scale=0.2},
			domain=3:260,
			samples=400
			]
			{2.07 * ln(x) + 0.88};
            			\addplot[
                mark options={scale=0.2},
				Plum,
				domain=3:260,
				samples=300
				]
				{ln(x)};
			\addlegendentry{\tiny{Data}}
			\addlegendentry{\tiny{Fit: $2.07\,\log p + 0.88$}}
			\addlegendentry{\tiny{$\ln p$}}
		\end{axis}
	\end{tikzpicture}
	\caption{Observed rank values and logarithmic model fit.}
    \label{graph:lowest rank}
\end{figure}

\subsection{Fermat and Mersenne Primes}
When analyzing the output of our computations for $\{2,p\}$ (collected in section \ref{table: max denominators 2,p}) we noticed a pattern for the maximal denominator used in a solution occurring for several special classes of primes. Upon further inspection, we were able to construct the explicit solutions that gave rise to these bounds. To explore these solutions, we make the following definitions. 
\begin{definition}[\cite{Pom}]
    Fermat primes are primes that are of the form $2^{k} + 1$ for some non-negative integer $k$. We denote the $n^{th}$ Fermat prime by $F_n$. 
\end{definition}

\begin{remark}
    For $2^{k} + 1$ to be prime, it is necessary that $k$ is a power of 2 (that is, $F_n = 2^{2^{n}} + 1$), which was observed by Fermat \cite{fermat1891}.  who conjectured that $2^n + 1$ is prime for any non-negative $n$. This conjecture was disproved by Euler for $k = 5$. It is currently known that for $5 \leq n \leq 32$ that $F_n$ is not prime, but it is an open question whether $F_4 = 65537$ is the largest Fermat prime. 
\end{remark}

We are concerned with these known Fermat primes, namely $F_0 = 3$, $F_1 = 5$, $F_2 = 17$, $F_3 = 257$, $F_4 = 65537$. Particularly, we noticed a pattern for $F_3$ in the table found in section \ref{table: max denominators 2,p} that lead us to the following proposition. 
\begin{proposition}
     When $p = F_n$ for $0 \leq n \leq 4$ then $[2, 4, .., 2^{2^{n}}, p, p^{2}, ..., p^{R-(2^{n} + 1)}, 2^{2^{n}}p^{R-(2^{n} + 1)}]_{2,p}$ is a solution to the Egyptian fraction problem in rank $R \geq 2^{n}+2$. 
\end{proposition}
\begin{proof}
    We fix n and define $k = 2^{n}$ for clarity. We will proceed by induction on $R$. Our base case is $R = k + 2$, in which our proposed solution is $[2, .., 2^{k}, p, 2^{k}p]$. First, recall
    $$\sum_{i = 1}^{k} \frac{1}{2^{i}} = \sum_{i = 1}^{k} \frac{2^{k-i}}{2^{k}} = \frac{1}{2^{k}}\sum_{i = 0}^{k-1} 2^{i} = \frac{2^{k} - 1}{2^{k}}$$

    Furthermore, $$\frac{2^k - 1}{2^k} + \frac{1}{p} = \frac{p(2^{k} - 1) + 2^{k}}{2^{k}p} = \frac{2^{k}p - 1}{2^{k}p},$$ 
where the last equality comes from $2^{k} + 1 = p$. So adding $\frac{1}{2^{k}p}$ gives us 1 and we have a solution for $R = k + 2$. 

    Now, assume for rank $R$, that $\frac{1}{2} + ...+\frac{1}{2^{k}} + \frac{1}{p} + ... + \frac{1}{p^{R-(k+1)}} + \frac{1}{2^{k}p^{R-(k+1)}} = 1$.  From this, we see that $$ \frac{1}{2} + ...+\frac{1}{2^{k}} + \frac{1}{p} + ... + \frac{1}{p^{R-(k+1)}} = \frac{2^{k}p^{R-(k+1)} - 1}{2^{k}p^{R-(k+1)} }$$
    So we have $\frac{1}{2^{k}p^{R - (k+1)}}$ to split between two terms for a rank $R + 1$ solution. We choose our $R^{th}$ denominator to be $p^{(R+1) - (k+1)}$. Now, 

    $$\frac{2^{k}p^{R-(k+1)} - 1}{2^{k}p^{R-(k+1)} } + \frac{1}{p^{R+1-(k+1)}} = \frac{p(2^{k}p^{R-(k+1)} - 1) + 2^{k}}{2^{k}p^{R+1-(k+1)} - 1} = \frac{2^{k}p^{R+1-(k+1)} - 1}{2^{k}p^{R+1-(k+1)}}$$
where again the last equality comes from the fact that $p = 2^{k}+ 1$.  So adding $\frac{1}{2^{k}p^{R+1-(k+1)}}$ gives us 1 as desired. 
\end{proof}

\begin{remark}
    As noted before, in some cases, $2^{2^{n}}p^{R-(2^{n} + 1)}$ is the maximum denominator that occurs in an Egyptian fraction solution at certain ranks (this pattern is especially easy to see for $F_2 = 17$).  For these ranks, it is not hard to see that $2^{2^{n}}p^{R-(2^{n} + 1)}$ is the bound produced by the greedy algorithm as well (to show this is essentially the proof above, with some justification to our choice of $p^{R+1-(k+1)}$ as the $R^{th}$ denominator in a rank $R + 1$ solution). 
    \newline
    
    However, $2^{2^{n}}p^{R-(2^{n} + 1)}$ is not the maximum denominator that appears in general, and we see this by looking at $p = 2$, $q = 5$ (so $n = 1$) at ranks greater than 8. We have also found that this pattern breaks for $q = 17$ in rank 16, where the maximum denominator is $2^{22}17^{7}$, and conjecture it ceases to give a maximum denominator for $F_3$ and $F_4$ as well (we don't have the computational power to find the breaking point for these primes). 
    \newline
    
    As an explanation for this disruption of our pattern, we notice that $2^25^3<2^9<5^4$, which would make our choice of $5^4$ as the next denominator no longer the greedy choice. Similarly, $2^417^{11} < 2^49 < 17^{12}$, which means $17^{12}$ is not the greedy choice in rank 16 . For $F_3 = 257$, we have that $2^8257^{177} < 2^{1425} < 257^{178}$ and so we conjecture that the breaking point for $F_3$ is somewhere around $R = 186$. 
\end{remark}

We now turn our attention to another special class of primes and show a similar result in this case. 
\begin{definition}[\cite{Pom}]
    A Mersenne prime is a prime $p$ that is $2^{n} - 1$ for some positive integer $n \geq 2$. In this case, we will denote $p$ as $M_n$. 
\end{definition}
\begin{remark}
    It is a necessary condition for $n$ to be prime for $M_n$ to be prime (this goes back at least to Euclid \cite{euclid1956}). As with the Fermat primes, this is not a sufficient condition. However, there are believed to be an infinite number of Mersenne primes. 
\end{remark}

We can again say something about the existence of certain Egyptian fraction solutions for every Mersenne prime. The following proposition was again deduced from studying  the table found in section \ref{table: max denominators 2,p}. 

Before we state our form of solution, we give the following lemma. Its correctness is easy to see by direct computation.  

\begin{lemma}
\label{useful Mersenne Lemma}
    If $p = 2^{q} + 1$, then $$ \frac{1 + \sum_{i = 1}^{q-2} 2^{i + (k-1)q}}{2^{kq - 1}p} - \frac{1}{2^{(k+1)q-1}} = \frac{1 + \sum_{i = 1}^{q-2} 2^{i + kq}}{2^{(k+1)q-1}p}.$$
\end{lemma}
\begin{proposition}
    \label{prop: mersenne prime solution}
    If $p = 2^{q} + 1$ is a Mersenne prime, then $$[2, 4, ..., 2^{q -1}, p, 2p, 4p, ...,  2^{q-2}p, 2^{2q-1}, 2^{3q-1}, ..., 2^{(r - (2q-1) + 1)q -1}, 2^{(r - (2q-1) + 1)q -1}p]_{2,p}$$ is a solution for rank $R$, where $R \geq 2q-1$. 
\end{proposition}
\begin{proof}
    To begin, observe that the sum of the first $q-1$ denominators of the sequence is $\frac{2^{q-1}-1}{2^{q-1}}$ and that adding $\frac{1}{p}$ gives us $\frac{(2^{q-1}-1)p + 2^{q-1}}{2^{q-1}p}$. We have used $q$ denominators thus far, and must find $R-q$ more denominators that, when inverted and summed, equal
    $$\frac{2^{q-1}p - ((2^{q-1}-1)p + 2^{q-1})}{2^{q-1}p} = \frac{1 + \sum_{i = 1}^{q-2} 2^{i}}{2^{q-1}p}.$$

    Notice that this fraction is in the form of the left hand side of Lemma \ref{useful Mersenne Lemma}, when $k = 1$. So we add $2^{2q-1}$ to our list of denominators, and the leftover is now exactly $\frac{1 + \sum_{i = 1}^{q-2} 2^{i + q}}{2^{2q-1}p}$. Also note that every term in the sum in the numerator individually divides the denominator, so we can break this into terms of $\frac{1}{i}$ with $i = 2p, 4p, ..., 2^{q-2}p, 2^{2q-1}p$. For $R = 2q-1$, this is a complete solution to the Egyptian fraction problem.

If $R > q +5$, we can repeat the step of adding $2^{kq-1}$ as terms until we have $q-1$ spots remaining, in which case we split the remaining fraction as in the second to last sentence of the preceding paragraph. This will give us a solution.
\end{proof}

Although the construction provided in the above proof is not the greedy construction, it turns out that the solution found is indeed the solution generated by the greedy algorithm. To prove this, we first state the following useful lemmas, in all of which we assume $p = 2^{q}-1$. 
\begin{lemma}
    \label{lemma: can't choose next power of 2}
    For any integer $i \geq 0$, 
    $$\frac{1}{2^{q+i}} > \frac{2^{q-(i+1)}-1}{2^{q-1}p}.$$
\end{lemma}
\begin{proof}
    First, notice that $p > 2^{q} - 2^{i+1}$ for any $i > 0$. Multiplying both sides of this inequality by $2^{q-1}$, we get $2^{q-1}p > 2^{2q-1} - 2^{q+i} = 2^{q+i}(2^{q-(i+1)} - 1)$. Dividing both sides by $2^{q-1}2^{q+i}p$ gives us the desired in equality. 
\end{proof}
\begin{lemma}
\label{lemma: 2^ip pattern}
For any integer $i > 0$, $$\frac{2^{q-i}-1}{2^{q-1}p} - \frac{1}{2^ip} = \frac{2^{q-i-1} - 1}{2^{q-1}p}.$$
\end{lemma}
\begin{proof}
    This is clear from direct computation, using $\frac{1}{2^ip} = \frac{2^{q-i-1}}{2^{q-1}p}$.
\end{proof}
\begin{lemma}
    \label{lemma: when is a power of 2 smallest}
    For any $q > 1, \,\, k >0$, both $2^{kq-1}p^2$ and $2^{kq}p$ are greater than $2^{(k+1)q-1}$.
\end{lemma}
\begin{proof}
   First,  $2^{kq}p = 2^{kq}(2^{q}-1)$ is greater than $2^{(k+1)q-1} = 2^{kq}(2^{q-1})$ when $2^{q}-1 > 2^{q-1}$ (exactly when $q > 1$). Next, notice $2^{(k+1)q-1} = 2^{kq-1}2^{q}$ and $2^{kq-1}p^2 = 2^{kq-1}(2^{2q} - 2^{q+1} +1)$. So we must only notice that $2^{q} < 2^{2q} - 2^{q+1} +1 = 2^{q}(2^{q}-2) + 1$ whenever $q > 1$ to show our desired inequality. 
\end{proof}
With the help of our lemmas, we are now ready to prove that the solution from   Proposition \ref{prop: mersenne prime solution} is the greedy one. 
\begin{proposition}
    The solution $$[2, 4, ..., 2^{q -1}, p, 2p, 4p, 8p, 2^{2q-1}, 2^{3q-1}, ..., 2^{(r - (q+4))q -1}, 2^{(r - (q+4))q -1}p]_{2,p}$$ is the sequence constructed by the greedy algorithm for primes 2 and $p$ in rank $R$. 
\end{proposition}
\begin{proof}
    As for any $p$, the first step at which the greedy algorithm does something non trivial is after we add $p$ to the list of denominators. At this point, as we noted in the proof of Proposition \ref{prop: mersenne prime solution}, the leftover space to fill is $\frac{2^{q-1}-1}{2^{q-1}p}$. Thus the greedy algorithm will try to add $\frac{1}{2^{q}}$, but  Lemma \ref{lemma: can't choose next power of 2}  with $i = 0$ tells us that adding this fraction will push our solution over 1. 
    So we instead must choose $2p$, which is the next biggest denominator (this is easy to see). Applying Lemma \ref{lemma: 2^ip pattern} with $i = 1$ we see that once we add $\frac{1}{2p}$ our remainder is $\frac{2^{q-2}-1}{2^{q-1}p}$. We continue by repeated applications of Lemmas \ref{lemma: can't choose next power of 2} and \ref{lemma: 2^ip pattern} until $i = q-2$, in which case, our leftover is $\frac{1}{2^{q-1}p}$. At this point, if our rank is $R = 2q-1$, we have constructed a solution. 

    If $R > 2q-1$, then we must choose the next smallest possible denominator bigger than $2^{q-1}p$. Our choices are $2^{2q-1}$ (which we use the fact that $p = 2^{q}-1$ to motivate), $2^{q}p$ and $2^{q-1}p^{2}$. By Lemma \ref{lemma: when is a power of 2 smallest}, we know that $2^{2q-1}$ is the smallest (take $k = 1$ in the lemma). Now, see that 
    $$\frac{1}{2^{q-1}p} - \frac{1}{2^{2q-1}} = \frac{2^{q}-p}{2^{2q-1}p} = \frac{1}{2^{2q-1}p}.$$
    This calculation is easily generalized. We can then use repeated applications of Lemma \ref{lemma: when is a power of 2 smallest} to choose the remainder of the denominators needed for a solution. 
\end{proof}
\begin{remark}
   The above argument tells us that this solution gives us the greedy algorithm bound for the maximum denominator for any Mersenne prime at any rank. However, as we have seen in Section \ref{sec: failure of greedy}, this greedy bound is not always the maximum denominator that appears. We expect this to be true after a certain point for all Mersenne primes. Anecdotally, for $p = 2$, $q = 7$ and $R = 10$ the maximum denominator is $2^{10}7^4$, which appears in a solution $[2,4,7,14,32,256, 2401, 7168, 614656, 2458624]_{2,7}$. We do not know how to determine when (both in general and for Mersenne primes) the greedy algorithm bound will not be the maximum denominator that appears in a solution. 
\end{remark}
\subsubsection{Other special primes}
It is possible (and perhaps even likely), that patterns of solutions like the ones presented in the above section exist for other special types of primes. We have calculated the following maximum bounds for \emph{Thabit primes} of the first and second kind (which are of the form $p = 3(2^n) -1$ and $p = 3(2^{n}) + 1$ respectively). The tables below present our findings. Note that no solutions exist for the pink cells, while the cells labeled  with  N/C and colored in green are ranks in which solutions exist, but could not be computed by the authors.

\begin{table}[H]
    \centering
\begin{tabular}{|c|c|c|c|c|c|c|c|c|}
    \hline
    & $R  = 12$ & $R = 13$ & $R = 14$ & $R = 15$ & $R = 16$ & $R = 17$ & $R = 18$ & $R = 19$ \\ \hline
         $q$ = 23 & $2^{12}23^3$ & $2^923^5$ & $2^{12}23^5$ & $2^923^7$ & \cellcolor{JungleGreen!50} N/C & \cellcolor{JungleGreen!50} N/C & \cellcolor{JungleGreen!50} N/C & \cellcolor{JungleGreen!50} N/C \\ \hline
         $q$ = 47 & $2^{16}47$ & $2^{19}47$ & $2^{13}47^3$ & $2^{17}47^3$ & $2^{21}47^3$ & \cellcolor{JungleGreen!50} N/C & \cellcolor{JungleGreen!50} N/C & \cellcolor{JungleGreen!50} N/C \\ \hline
         
         $q$ = 191 & \cellcolor{WildStrawberry!50} & \cellcolor{WildStrawberry!50} & $2^{13}191$ & $2^{19}191$ & $2^{25}191$ & $2^{27}191$ &   $2^{33}191$ & \cellcolor{JungleGreen!50} N/C\\ \hline

          $q$ = 383 & \cellcolor{WildStrawberry!50} & \cellcolor{WildStrawberry!50}& \cellcolor{WildStrawberry!50}& \cellcolor{WildStrawberry!50} & $2^{15}383$ & $2^{22}383$ &   $2^{29}383$ & $2^{31}383$\\ \hline
    \end{tabular}
    \caption{Maximum Denominators for Thabit primes of the first kind}
    \label{tab:first kind thabit primes}
\end{table}
\begin{table}[H]
    \centering
     \resizebox{\textwidth}{!}{%
    \begin{tabular}{|c|c|c|c|c|c|c|c|c|c|c|}
    \hline
    & $R = 9$ &$R = 10$ & $R = 11$&  $R  = 12$ & $R = 13$ & $R = 14$ & $R = 15$ & $R = 16$ & $R = 17$ & $R = 18$  \\ \hline
         $q$ = 97 & $2^697$ &$2^797$& $2^697^2$ &$2^797^2$&$2^697^2$ &$2^797^2$&$2^697^2$ &$2^{17}97^3$& \cellcolor{JungleGreen!50} N/C & \cellcolor{JungleGreen!50} N/C\\ \hline
         $q$ = 193 & \cellcolor{WildStrawberry!50}& $2^{7}193$ & $2^8193$ & $2^{7}193^2$ & $2^8193^2$ &$2^{7}193^3$ & $2^8193^3$ &$2^{7}193^3$ & \cellcolor{JungleGreen!50} N/C & \cellcolor{JungleGreen!50} N/C \\ \hline
         
         $q$ = 769 & \cellcolor{WildStrawberry!50} &\cellcolor{WildStrawberry!50}& \cellcolor{WildStrawberry!50} & $2^{9}769$ & $2^{10}769$ & $2^{9}769^2$ & $2^{10}769^2$ & $2^{9}769^3$ & $2^{10}769^3$ & $2^9769^4$\\ \hline
    \end{tabular}}
    \caption{Maximum Denominators for Thabit primes of the second kind}
    \label{tab:second kind thabit primes}
\end{table}

\subsection{Maximal denominators for $\{2,p\}$}
\label{table: max denominators 2,p}
To investigate a possible bound on the maximal denominator appearing in solutions for the set $\{2, p\}$, we present the following table. The data was generated using the code given in Appendix \ref{sec:code-two-primes}. We note that no solutions exist for the pink-colored cells.

\begin{table}[H]
    \centering
    \begin{tabular}{|c|c|c|c|c|c|c|c|}
        \hline
         & $R=5$ & $R=6$ & $R=7$ & $R=8$ & $R=9$ & $R=10$ & $R=11$ \\
        \hline
        $p=3$ & $2^33^3$ & $2^33^5$ & $2^33^7$ & $2^{10}3^6$ & $2^{10}3^9$ & $2^{12}3^{11}$ & {$2^{15}3^{14}$} \\
        \hline
        $p=5$ & $2^25^2$ & $2^25^3$ & $2^25^4$ & $2^{10}5^3$ & $2^95^5$ & $2^{12}5^5$ & $2^{10}5^{9}$ \\
        \hline
        $p = 7$ & $2^27$ & $2^57$ & $2^87$ & $2^{11}7$ & $2^87^3$ & $2^{10}7^4$ & $2^{11}7^{5}$ \\
        \hline
        $p=11$ & \cellcolor{WildStrawberry!50} & $2^511$ & $2^611$ & $2^511^3$ & $2^611^3$ & $2^511^5$ & $2^{13} 11^4$ \\
        \hline
        $p=13$ & \cellcolor{WildStrawberry!50} & $2^313$ & $2^613$ & $2^713$ & $2^613^3$ & $2^{11}13^2$ & $2^6 13^5$  \\
        \hline
        $p=17$ & \cellcolor{WildStrawberry!50} & $2^417$ & $2^417^2$ & $2^417^3$ & $2^417^4$ & $2^417^5$ & $2^417^6$ \\
        \hline
        $p=19$ & \cellcolor{WildStrawberry!50} & \cellcolor{WildStrawberry!50} & $2^419$ & $2^919$ & $2^{10}19$ & $2^919^2$ & $2^{12 }19^3$ \\
\hline
    \end{tabular}
    \caption{Maximal denominators appearing in rank $R$ for $\{2,p\}$}
    \label{tab:placeholder_label}
\end{table}

We illustrate the growth of these denominators in the following graph. We used a logarithmic scale on the vertical axis for better visual interpretation, because the values $2^a p^b$ grow exponentially fast. We observe that all curves trend upward, and that smaller primes like $p=3, 5$ have faster growth. 

	\begin{figure}[H]
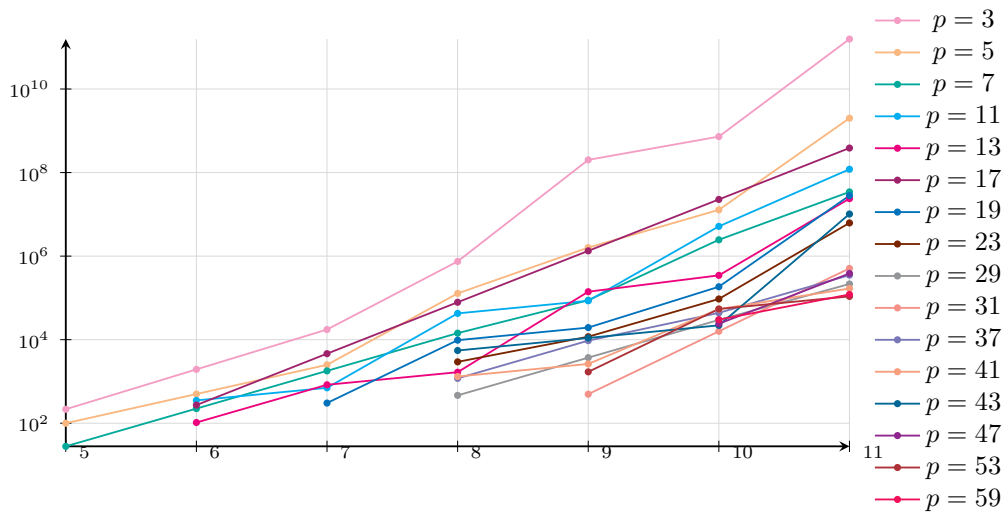

		\centering
		%
		\caption{Value of $2^a \cdot p^b$ vs.\ $R$ (fixed $p$).}
	\end{figure}

Below we test possible lower bounds that depend on $2$, $p$, and $R$, in the hopes of finding a general formula. The dashed lines represent candidate functions of the form $c^{R - d}$ that approximate the observed growth. 
The behavior seems to vary significantly with the size of the prime, suggesting that the constants used in a good general bound may depend on how far $p$ is from $2$, or from some power of $2$ relative to the rank.

    	
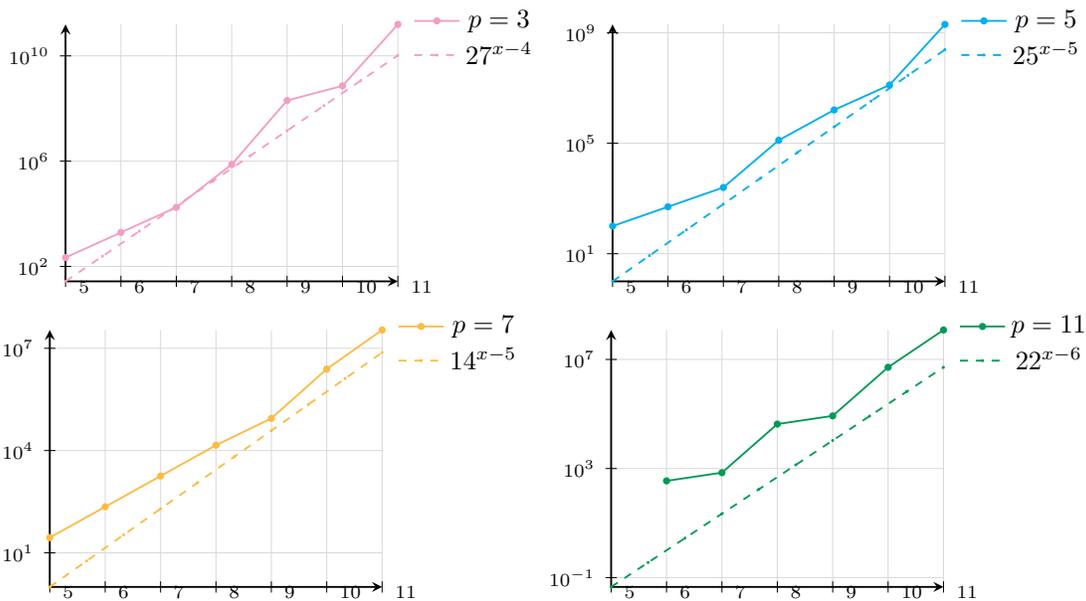
\begin{figure}[H]
		%
	
	\end{figure}

\subsection{Number of Solutions for $\{2,p\}$}

In the following table, we count the number of solutions for a fixed set \{2,p\}, categorized by rank. The code used to generate this data is provided in Appendix \ref{sec:low rank code}.

\begin{table}[H]
\small
    \centering
 \resizebox{\textwidth}{!}{%
    \begin{tabular}{|c|c|c|c|c|c|c|c|c|c|c|c|c|c|}
        \hline
        $q$ = & $R$ = 3  & $R$ = 4 & $R$ = 5 & $R$ = 6 & $R$ = 7 & $R$ = 8 & $R$ = 9 & $R$ = 10 & $R$= 11 & $R$ = 12 & $R$ = 13 & $R$ = 14 & $R$ = 15 \\ \hline
        3 & 3 & 10 & 44 & 227 & 1343 & 9082 & 69394 & 595273 &  \cellcolor{JungleGreen!50} N/C& \cellcolor{JungleGreen!50} N/C & \cellcolor{JungleGreen!50} N/C & \cellcolor{JungleGreen!50} N/C& \cellcolor{JungleGreen!50} N/C\\ \hline
        5 & \cellcolor{WildStrawberry!50} & 4 & 13 & 45 & 184 & 848 & 4420 & 25882 & 168712 & \cellcolor{JungleGreen!50} N/C& \cellcolor{JungleGreen!50} N/C&\cellcolor{JungleGreen!50} N/C & \cellcolor{JungleGreen!50} N/C\\ \hline
        7 & \cellcolor{WildStrawberry!50} & \cellcolor{WildStrawberry!50} & 5 & 18 & 68 & 269 & 1155 & 5369 & 27547 & 155245& \cellcolor{JungleGreen!50} N/C & \cellcolor{JungleGreen!50} N/C & \cellcolor{JungleGreen!50} N/C\\ \hline
        11 &\cellcolor{WildStrawberry!50}  & \cellcolor{WildStrawberry!50} & \cellcolor{WildStrawberry!50} & 9 & 35 & 119 & 449 & 1844 & 8213 & 40458& 217035 & \cellcolor{JungleGreen!50} N/C & \cellcolor{JungleGreen!50} N/C \\ \hline
        13 & \cellcolor{WildStrawberry!50} & \cellcolor{WildStrawberry!50} & \cellcolor{WildStrawberry!50} & 7 & 22 & 73 & 255 & 973 & 4058 & 18404& 90387 &478610 & \cellcolor{JungleGreen!50} N/C\\ \hline
        17 & \cellcolor{WildStrawberry!50} & \cellcolor{WildStrawberry!50} & \cellcolor{WildStrawberry!50} & 7 & 23 & 74 & 241 & 820 & 2940 & 11321& 47428 &218309 & 1100840\\ \hline
        19 & \cellcolor{WildStrawberry!50} & \cellcolor{WildStrawberry!50} & \cellcolor{WildStrawberry!50} & \cellcolor{WildStrawberry!50}& 11 & 37 & 137 & 490 & 1828 & 7406& 32698 &115942 & 795516 \\ \hline
         
    \end{tabular}}
    \caption{Number of solutions for $\{2,p\}$ by rank}
    \label{tab:my_label}
\end{table}

To picture how the growth varies for a fixed prime $p$, we include the following graph.
\begin{figure}[H]
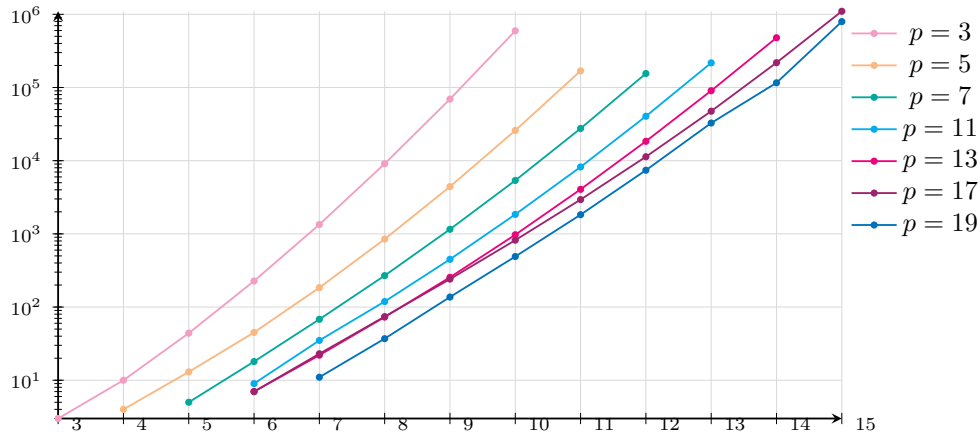

		\centering
%
	\caption{Number of solutions for $\{2,p\}$ by rank.}
	\end{figure}

Here, curves appear exponential on a logarithmic scale, suggesting that the growth in the number of solutions may be exponential or super-exponential. We note that smaller primes (like $q=3$, 5 and 7) have a faster growth for increasing $R$.

\section{Future directions}

We conclude by highlighting some of the questions that arose during this project, which we believe may be of interest.

\begin{question}
    In Corollary \ref{cor: gcd condition}, we give a necessary condition for finding a solution at rank $R$. Can we find a sufficient condition? If so, can we give an explicit construction of a solution for ranks satisfying this condition?
\end{question}

For instance, for the set $\{2,p\},$ if we impose the restriction on $R$ so that $2^{R-3}<p$, and $p-2^{R-3}=2^{l+1}$ for some $l$, then $[2,2^4,2^8, \dots, 2^{R-3}, p, 2^{R-l-3}p, 2^{R-3}p]_{2,p}$ 
is a solution at rank $R$.

\begin{question}
    In Remark \ref{remk: bound on R} we give a trivial bound on $R$ for the existence of a solution. For a fixed set of primes, can we improve this bound? 
\end{question}

We note the question above is related to Table \ref{tab: lowest rank label} and Figure \ref{graph:lowest rank}. It could be interesting to test other sets of two primes $\{p,q\}$ to see if there is a recognizable pattern for the lowest rank in terms of $p$ and $q$. 
\begin{question}
   Graph \ref{graph:lowest rank} seems to indicate that $\ln(p)$ (and even $\ln(2p)$) is a lower bound for the rank of solutions with $\{2,p\}$. Can we prove this bound, and/or improve it? Can a similar bound be found (maybe $\ln(pq)$) for solutions with $\{p,q\}$?
\end{question}

It would of course be interesting to find explicit solutions for the lowest possible bound as well. 

\begin{question}
For a fixed rank $R$, what is the maximal possible difference $|p-q|$ for there to exist a solution $[x_1, \dots, x_R]_{p,q}$?
\end{question}

\begin{question}
    Table \ref{tab:number of solutions one prime} gives the number of solutions with solely the divisor $p$ at rank $R = (p-1)t +1$. For all $p$, the number of solutions begins as $\{1,1, 2^2, ..., 2^{t-2}\}$ and for small $p$ we observe from our computations that at some point $\mathfrak{s}_p(R)$ diverges from this pattern. We suspect this to be true for all primes $p$. For a fixed prime $p$, can we say at which rank $R$, the pattern breaks? If so, can we also determine the value $\mathfrak{s}_p(R)$?
\end{question}

   \bibliography{bib} 
 \bibliographystyle{alpha}
    
\begin{appendices}

\include{one}

\include{two}

\include{low}

\include{Solsone}

\include{Sols5}
\include{Sols6}

\end{appendices}

\end{document}

%% file: one.tex
\section{Code for one prime}
\label{appendix: one prime code}
Below are the relevant lines of code written in Sage for computing solutions with one prime divisor. In part, this code implements an algorithm introduced in \cite[Algorithm 4.1]{BR} which describes a way to construct all solutions in a fixed rank $R$. 

\begin{lstlisting}
from sage.functions.log import logb
from enum import Enum
Status = Enum('Solution_Status', ['INVALID', 'VALID', 'INCOMPLETE'])

# This section contains code for computations related to solutions involving one prime divisor

# NOTE: we have set p and rank to temporary values here.

# GLOBAL VARIABLES: set p and rank for the duration of this code

p = 7
rank = 6*4+1

MAX_LENGTH = rank-1
TERMINAL_POWER = MAX_LENGTH/(p-1)
# By Proposition 2.12, taking the highest power of p that appears in a candidate solution to be (rank - 1)/(p-1) is sufficient to find all solutions

# For efficiency, we create a list possible_denoms of all permissable denominators in a solution given p and rank and a list nice_sums which is used to check when a list of (rank-1)-many values can be completed to a solution

possible_denoms = []

for i in range(0,TERMINAL_POWER + 1):
    possible_denoms.append(p^i)
        
possible_denoms.sort()
possible_denoms.remove(1)

nice_sums = [(n-1)/n for n in possible_denoms]

# HELPER FUNCTIONS

# is_solution determines whether a list is a solution, is not a solution, or is too short of a list to be a complete solution
# __find_solutions_loop is used to help create a list of all possible solutions

def is_solution(candidate_sol):
    if len(candidate_sol) > MAX_LENGTH: 
        return Status.INVALID
    elif sum([(1/x_i) for x_i in candidate_sol]) > 1:
        return Status.INVALID
    # condition to return true (need to add to solutions list)
    elif len(candidate_sol) == MAX_LENGTH and sum([(1/x_i) for x_i in candidate_sol]) in nice_sums:
        return Status.VALID
\end{lstlisting}
\newpage
\begin{lstlisting}
    else:
        return Status.INCOMPLETE

def __find_solutions_loop(partial_sol, solutions_list):
    match is_solution(partial_sol):
        case Status.INVALID:
            return
        case Status.VALID:
            solutions_list.append(list(partial_sol))
        case Status.INCOMPLETE:
            # loop through denominators, check if its bigger (or equal)
            for i in possible_denoms: 
            
                if len(partial_sol) == 0 or i >= partial_sol[-1]:
                    #add to list, recursively call function
                    partial_sol.append(i)
                    __find_solutions_loop(partial_sol, solutions_list)
                
                    #remove from list
                    partial_sol.remove(i)

    return solutions_list

\end{lstlisting}

\begin{lstlisting}
# MAIN CODE: Functions computing solutions and their properties

# find_solutions returns a list of all size rank-1 lists that can be completed to a solution
# find_possible_dim_C returns the largest x_i in each solution
# find_max_power returns log_p(x_i) of the largest x_i in each solution
# find_full_soln_without_duplication returns a list of all unique solutions (in particular, unlike in find_solutions, these solutions have length=rank)
# find_number_of_soln returns the number of unique solutions at the given prime and rank

def find_solutions(init_partial_sol = []):
    return __find_solutions_loop(init_partial_sol, [])

def find_possible_dim_C():
    denoms = []
    for solution in find_solutions([]):
        x = 0 
        for i in solution:
            x=x+1/i
        denoms.append(x.denominator())
    return denoms

def find_max_power(prime):
    denoms = []
    for solution in find_solutions([]):
        x = 0 
        for i in solution:
            x=x+1/i
        denoms.append(logb(x.denominator(), prime))  
    denoms.sort()
    return denoms     
\end{lstlisting}
\newpage
\begin{lstlisting}
def find_full_soln_without_duplication():
    full_solutions = find_solutions()
    
    # Append final denominator to get a full suloution
    for sol in full_solutions:
        s = 0
        for i in sol:
            s += 1/i
        sol.append(s.denominator())

    # Create a dictionary that tracks how many times some value x_i shows up in each solution
    solutions_dictionary =[]

    for solution in full_solutions:
        sol_dict = {}
        for i in solution:
            if i in sol_dict.keys():
                sol_dict[i] += 1
            else:
                sol_dict[i] = 1

        solutions_dictionary.append(sol_dict)

    # Creates a new list of the unique solutions    
    dup_free = [solutions_dictionary[0]]

    for i in range(len(solutions_dictionary)):
            unique = True
            for item in dup_free:

                if solutions_dictionary[i] == item:

                    unique = False
            if unique:

                dup_free.append(solutions_dictionary[i])

    return dup_free

def find_number_of_soln():
    return len(find_full_soln_without_duplication())

\end{lstlisting}

\vspace{1em}
\noindent Implementation Notes: A Jupyter notebook containing this code can be found in the accompanying \href{https://github.com/aswatkin/Remarks-on-Egyptian-Fractions-For-Few-Primes.git}{GitHub repository}. The code provided above is written and run in Sage 9.7.

%% file: two.tex
\section{Code for two primes}\label{sec:code-two-primes}

We give below our code developed to find all solutions for a given $p$, $q$ and rank.

\begin{lstlisting}
from fractions import Fraction

# NOTE: we have set p,q, and rank to temporary values here. 

# GLOBAL VARIABLES - set p, q, and rank for the duration of this code

p = 2
q = 13
rank = 7

# highest_power limits the denominators that are considered. 
highest_power = 20 

#initialize a variable for the list of solutions that will be used throughout the code. 
list_of_solutions = []

#HELPER FUNCTIONS 

# sum_list(denom_list) --- sums the reciprocals of a list of numbers
# Input: denom_list, a list of integers
# Output: running_frac, the sum of the reciprocals of the entries in denom_list

def sum_list(denom_list):
    
    running_frac = 0
    for denom in denom_list:
        running_frac += 1/denom
    
    return running_frac

# get_possible_denoms(p,q) - generating a list of p^iq^j in increasing order
# Input: p, q, two prime numbes
        highest_power, a limit for i and j
# Output: possible_denoms, a list of p^iq^j for i,j <= highest_power. 

def get_possible_denoms(p,q, highest_power):
    
    possible_denoms = []
    
    for i in range(0, highest_power):
    
        for j in range(0, highest_power):
        
            if i + j > 0:
                possible_denoms.append(p**i * q **j)
            
    possible_denoms.sort()

    return possible_denoms
\end{lstlisting}
\newpage
\begin{lstlisting}
# FUNCTIONS THAT IMPLEMENT ALGORITHMS

# get_proposed_bound(prime_1, prime_2, rank, highest_power) - implements the greedy algorithm for conjecturing a maximal denominator.
# Input: prime_1, prime_2, two primes (integers) appearing in solutions
        rank, the rank (integer) at which to find solutions
        highest_power, an integer used for a call to get_possible_denoms()

# Output: an integer giving the greedy algorithm result. 
def get_proposed_bound(prime_1, prime_2, rank, highest_power):
    possible_denoms = get_possible_denoms(prime_1, prime_2, highest_power)
    
    partial_solution = []
    total_sum = 0 
    for denominator in possible_denoms:
        if total_sum + 1/denominator < 1 and len(partial_solution) < rank - 1:
            new_sum = total_sum + 1/denominator
            total_sum += 1/denominator
            partial_solution.append(denominator)

    left_over = 1 - total_sum
    simplified_left_over = 1 / left_over
    proposed_bounds = 0

    counter = 0
    while possible_denoms[counter] < simplified_left_over:
        counter += 1

    return possible_denoms[counter]

# get_solutions(running_sol, possible_denoms) - implements the recursive algorithm found in section \ref{subsection: recursive algorithm}.
# Input: running_sol, a list with a partial solution to the Egyptian fraction problem. 
        possible_denoms, the denominators to consider for this solution. 
# Output: None, this function populates a global variable containing a list of solutions. 

def get_solutions(running_sol, possible_denoms):
    
    running_sum = sum_list(running_sol)
    
    #BASE CASE BLOCK
    
    #BC 1 - we've gone over 1 somehow
    if running_sum > 1: 
        return None
    #BC 2 - once the sum is 1, we should stop
    elif running_sum == 1:
        # if we have enough 
        if len(running_sol) == rank:
            list_of_solutions.append(running_sol[ : ]) # make a deep copy
            return None
        else:
        \end{lstlisting}
\newpage
\begin{lstlisting}
            
            return None
    elif len(running_sol) >= rank:
       
        return None
    else: #this should be all lists that are shorter than rank and sum to less than 1
        
        #find the right data for iteration
       
        m_n_frac = 1 - running_sum
        m = m_n_frac.numerator()
        n = m_n_frac.denominator()
        s = len(running_sol)
        
        # get the index of where we're currently at in the list
        if len(running_sol) >= 1: 
            index = possible_denoms.index(running_sol[-1])
        else:
            index = 0
        
        # calculate the biggest possible denominator
        biggest_denom = (n * (rank - s))/m
        
        # using a while loop instead of just iterating over all possible denominators eliminates a bunch of sequences
        # that will never sum to 1
        while possible_denoms[index] <= biggest_denom:
            
            running_sol.append(possible_denoms[index])
            
            #recursive call
            get_solutions(running_sol, possible_denoms)
            
            running_sol.pop()
            index += 1
\end{lstlisting}
\newpage
\begin{lstlisting}
    
#MAIN CODE - gets solutions and prints info for p,q and rank

#BLOCK 1 - find all solutions and greedy bound
possible_denoms = get_possible_denoms(p,q, highest_power)

greedy_bound = get_proposed_bound(p,q,rank, highest_power)

# This line ensures that we are not storing solutions from previous times running this code.
list_of_solutions.clear()

get_solutions([], possible_denoms)

# BLOCK 2 - find and describe maximum denominator as p^iq^j
max_denominator = 0

for solution in list_of_solutions:
    for denominator_test in solution:
        if denominator_test > max_denominator:
            max_denominator = denominator_test

#factor all powers of two out

to_factor = max_denominator
exp_p = 0

while (to_factor/p).is_integer():
    to_factor = to_factor/p
    exp_p += 1
    
exp_q = 0

while (to_factor/q).is_integer():
    to_factor = to_factor/q
    exp_q += 1

#BLOCK 3 - display info about solutions

#UN-COMMENT THE FOLLOWING TWO LINES TO PRINT ALL SOLUTIONS
#for solution in list_of_solutions: 
  #print(solution)
  
print("There are", len(list_of_solutions), "solutions")

print("The maximum denominator that appears is: ", max_denominator, "which is", str(p)+"^"+str(exp_p), str(q)+"^"+str(exp_q))
if max_denominator <= greedy_bound:
    print("This is within the greedy algorithm bound", greedy_bound)
else:
    print("WARNING: this does not agree with the greedy bound", greedy_bound)
\end{lstlisting}

\vspace{1em}
\newpage
\noindent The following is the output when this code is run in a Jupyter notebook with a Sage 10.6 Kernel:

\begin{lstlisting}
There are 22 solutions
The maximum denominator that appears is:  832 which is 2^6 13^1
WARNING: this does not agree with the greedy bound 676
\end{lstlisting}
\vspace{1em}
\noindent Implementation Notes: A Jupyter notebook containing this code can be found in the accompanying \href{https://github.com/aswatkin/Remarks-on-Egyptian-Fractions-For-Few-Primes.git}{GitHub repository}. In addition, a near identical implementation in Julia can also be found in this repository. The code provided above is written and run in Sage 10.6. 
\newline

The following table gives the time (in seconds) for the recursive part of the algorithm to run for $p = 2$ on a MacBook Air with an Apple M1 chip and 8GBs of memory.

\begin{table}[H]
\footnotesize
    \centering
     \resizebox{\textwidth}{!}{%
    \begin{tabular}{|c|c|c|c|c|c|c|c|c|c|c|c|c|c|}
        \hline
        $q$ = & $R$ = 3  & $R$ = 4 & $R$ = 5 & $R$ = 6 & $R$= 7 & $R$ = 8 & $R$ = 9 & $R$ = 10 & $R$ = 11 & $R$ = 12 & $R$ = 13 & $R$ = 14 & $R$ = 15 \\ \hline
        3 & $5\cdot10^{-5}$ & .00018 & .00075 & .0059 & .0542 & .5902 & 7.484 & 110.54 & \cellcolor{JungleGreen!50} N/C& \cellcolor{JungleGreen!50} N/C & \cellcolor{JungleGreen!50} N/C & \cellcolor{JungleGreen!50} N/C& \cellcolor{JungleGreen!50} N/C 
        \\ \hline
        
        5 & \cellcolor{WildStrawberry!50} & $6.1\cdot10^{-5}$ & .00026 & .0011 & .0071 & .0515 & .4487 & 4.35 & 46.895 & \cellcolor{JungleGreen!50} N/C& \cellcolor{JungleGreen!50} N/C&\cellcolor{JungleGreen!50} N/C & \cellcolor{JungleGreen!50} N/C\\
        
        \hline
        7 & \cellcolor{WildStrawberry!50} & \cellcolor{WildStrawberry!50} & .00014 & .00052 & .0026 & .0145 & .0905 & .6765 & 5.666 & 54.09 & \cellcolor{JungleGreen!50} N/C & \cellcolor{JungleGreen!50} N/C & \cellcolor{JungleGreen!50} N/C\\ \hline

        11 &\cellcolor{WildStrawberry!50}  & \cellcolor{WildStrawberry!50} & \cellcolor{WildStrawberry!50} & .0003 & .0013 & .006 & .0337 & .2244 & 1.629 & 12.56 & 112.04 & \cellcolor{JungleGreen!50} N/C & \cellcolor{JungleGreen!50} N/C \\ \hline

        13 & \cellcolor{WildStrawberry!50} & \cellcolor{WildStrawberry!50} & \cellcolor{WildStrawberry!50} & .00022 & .00091 & .0038 & .0187 & .1058 & .677 & 4.822& 38.237 & 333.68 & \cellcolor{JungleGreen!50} N/C\\ \hline
        
        17 & \cellcolor{WildStrawberry!50} & \cellcolor{WildStrawberry!50} & \cellcolor{WildStrawberry!50} & .0016 & .0007 & .0032 & .0148 & .074 & .4148 & 2.711& 20.76 & 177.231 & 1610.01\\ \hline
       
        19 & \cellcolor{WildStrawberry!50} & \cellcolor{WildStrawberry!50} & \cellcolor{WildStrawberry!50} & \cellcolor{WildStrawberry!50}& .0006 & .0021 & .0097 & .0494 & .272 & 1.785& 12.67 & 101.42 & 836.299 \\ \hline
         
    \end{tabular}}
    \caption{Run time for $p = 2$ with given $q$ and $R$}
    \label{tab:my_label}
\end{table}

%% file: low.tex
\section{Code for low rank}\label{sec:low rank code}

Here is some Maple code for Algorithm \ref{alg:low rank}.

\begin{lstlisting}
 generate_A := proc(n::posint)
    local u, A, j, X, k, fx, i, newlist, elem;

    # Step 1: Define auxiliary sequence u[1..n]
    u := Vector(n):
    u[1] := 1:
    for i from 2 to n do
        u[i] := u[i-1]*(u[i-1]+1);
    end do;

    # Step 2: Initialize A[1]
    A := Array(1..n):
    A[1] := [seq([k], k=2..n)];

    # Step 3: Build A[2] to A[n-1]
    for j from 2 to n-1 do
        newlist := [];
        for elem in A[j-1] do
            for k from elem[-1] to u[j]*(n-j+1) do
                if add(1/elem[i], i=1..nops(elem)) + 1/k < 1 then
                    newlist := [op(newlist), [op(elem), k]];
                end if;
            end do;
        end do;
        A[j] := newlist;
    end do;

    # Step 4: Build A[n]
    newlist := [];
    for elem in A[n-1] do
        fx := 1 - add(1/elem[i], i=1..nops(elem));
        if fx > 0 and type(1/fx, integer) and 1/fx >= elem[-1] then
            newlist := [op(newlist), [op(elem), 1/fx]];
        end if;
    end do;
    A[n] := newlist;

    # Output A[n]
    return A[n];
end proc:

\end{lstlisting}

%% file: Solsone.tex
\section{Solutions library: one prime}\label{App: solutions: one prime}

We list here the full list of solutions to Equation \eqref{eq: one prime} for a fixed prime at low ranks. The code used for generating these solutions is provided in Appendix \ref{appendix: one prime code}.

At each table, the set $\{a_1:b_1, \dots, a_l:b_l\}$ denotes a solution of the form 
\begin{align}
		\left[  \underbrace{a_1, \dots, a_1}_{b_1}, \dots, \underbrace{a_l, \dots, a_l}_{b_l}\right]_p.
	\end{align}
We also color the solutions showcasing the highest power of the given prime that appears at each rank, in reference to Proposition \ref{prop: bound for one rank}.

\vspace{1cm}
\small

\begin{longtable}[h]{ | p{0.06\textwidth} | >{\raggedright\arraybackslash}p{0.84\textwidth} |}
\hline
Rank &  Solutions for $p=3$  \\
\hline
3 & 
$\{\textcolor{purple}{3}:3\}$\\
\hline
5 & 
$\{3: 2, \textcolor{purple}{3^2}:3\}$\\
\hline
7 & $\{3: 2, 3^2:2, \textcolor{purple}{3^3}:3\}, \{3:1, 3^2:6\}$\\
\hline
9 &  $\{3: 2, 3^2: 2, 3^3: 2, \textcolor{purple}{3^4}: 3\}, \{3: 2, 3^2: 1, 3^3: 6\}, \{3: 1, 3^2: 5, 3^3: 3\}, \{3^2: 9\}$ \\
\hline
11 & $ \{3: 2, 3^2: 2, 3^3: 2, 3^4: 2, \textcolor{purple}{3^5}: 3\},
 \{3: 2, 3^2: 2, 3^3: 1, 3^4: 6\},
 \{3: 2, 3^2: 1, 3^3: 5, 3^4: 3\},
 \{3: 2, 3^3: 9\},
 \{3: 1, 3^2: 5, 3^3: 2, 3^4: 3\},
 \{3: 1, 3^2: 4, 3^3: 6\},
 \{3^2: 8, 3^3: 3\}$\\
\hline
13 &  $\{3: 2, 3^2: 2, 3^3: 2, 3^4: 2, 3^5: 2, \textcolor{purple}{3^6}: 3\}, \{3: 2, 3^2: 2, 3^3: 2, 3^4: 1, 3^5: 6\}, \{3: 2, 3^2: 2, 3^3: 1, 3^4: 5, 3^5: 3\}, \{3: 2, 3^2: 2, 3^4: 9\}, \{3: 2, 3^2: 1, 3^3: 5, 3^4: 2, 3^5: 3\}, \{3: 2, 3^2: 1, 3^3: 4, 3^4: 6\}, \{3: 2, 3^3: 8, 3^4: 3\}, \{3: 1, 3^2: 5, 3^3: 2, 3^4: 2, 3^5: 3\}, \{3: 1, 3^2: 5, 3^3: 1, 3^4: 6\}, \{3: 1, 3^2: 4, 3^3: 5, 3^4: 3\}, \{3: 1, 3^2: 3, 3^3: 9\}, \{3^2: 8, 3^3: 2, 3^4: 3\}, \{3^2: 7, 3^3: 6\}$  \\
\end{longtable}

\newpage

\begin{longtable}[h]{ | p{0.06\textwidth} | >{\raggedright\arraybackslash}p{0.84\textwidth} |}
\hline
Rank &  Solutions for $p=5$  \\
\hline
5 & 
$\{\textcolor{purple}{5}:5\}$\\
\hline
9 & $\{5: 4, \textcolor{purple}{5^2}: 5\}$\\
\hline
13 &  $ \{5: 4, 5^2: 4, \textcolor{purple}{5^3}: 5 \},  \{5: 3, 5^2: 10 \}$ \\
\hline
17 & $  \{5: 4, 5^2: 4, 5^3: 4, \textcolor{purple}{5^4}: 5 \},  \{5: 4, 5^2: 3, 5^3: 10 \},  \{5: 3, 5^2: 9, 5^3: 5 \},  \{5: 2, 5^2: 15 \}$\\
\hline
21 &  $ \{5: 4, 5^2: 4, 5^3: 4, 5^4: 4, \textcolor{purple}{5^5}: 5 \},  \{5: 4, 5^2: 4, 5^3: 3, 5^4: 10 \},  \{5: 4, 5^2: 3, 5^3: 9, 5^4: 5 \},  \{5: 4, 5^2: 2, 5^3: 15 \},  \{5: 3, 5^2: 9, 5^3: 4, 5^4: 5 \},  \{5: 3, 5^2: 8, 5^3: 10 \},  \{5: 2, 5^2: 14, 5^3: 5 \},  \{5: 1, 5^2: 20 \}$  \\
\hline
25 & $ \{5: 4, 5^2: 4, 5^3: 4, 5^4: 4, 5^5: 4, \textcolor{purple}{5^6}: 5 \},  \{5: 4, 5^2: 4, 5^3: 4, 5^4: 3, 5^5: 10 \},  \{5: 4, 5^2: 4, 5^3: 3, 5^4: 9, 5^5: 5 \},  \{5: 4, 5^2: 4, 5^3: 2, 5^4: 15 \},  \{5: 4, 5^2: 3, 5^3: 9, 5^4: 4, 5^5: 5 \},  \{5: 4, 5^2: 3, 5^3: 8, 5^4: 10 \},  \{5: 4, 5^2: 2, 5^3: 14, 5^4: 5 \},  \{5: 4, 5^2: 1, 5^3: 20 \},  \{5: 3, 5^2: 9, 5^3: 4, 5^4: 4, 5^5: 5 \},  \{5: 3, 5^2: 9, 5^3: 3, 5^4: 10 \},  \{5: 3, 5^2: 8, 5^3: 9, 5^4: 5 \},  \{5: 3, 5^2: 7, 5^3: 15 \},  \{5: 2, 5^2: 14, 5^3: 4, 5^4: 5 \},  \{5: 2, 5^2: 13, 5^3: 10 \},  \{5: 1, 5^2: 19, 5^3: 5 \},  \{5^2: 25 \}$\\
\end{longtable}

\small
\begin{longtable}[h]{ | p{0.06\textwidth} | >{\raggedright\arraybackslash}p{0.84\textwidth} |}
\hline
Rank &  Solutions for $p=7$  \\
\hline
\small
7 & 
$\{\textcolor{purple}{7}:1\}$\\
\hline
13 & $\{7: 6, \textcolor{purple}{7^2}: 7\}$\\
\hline
19 &  $ \{7: 6, 7^2: 6, \textcolor{purple}{7^3}: 7\}, \{7: 5, 7^2: 14\}$ \\
\hline
25 & $  \{7: 6, 7^2: 6, 7^3: 6, \textcolor{purple}{7^4}: 7\}, \{7: 6, 7^2: 5, 7^3: 14\}, \{7: 5, 7^2: 13, 7^3: 7\}, \{7: 4, 7^2: 21\}$\\
\hline
\small 31 & $\{7: 6, 7^2: 6, 7^3: 6, 7^4: 6, \textcolor{purple}{7^5}: 7\}, \{7: 6, 7^2: 6, 7^3: 5, 7^4: 14\}, \{7: 6, 7^2: 5, 7^3: 13, 7^4: 7\}, \{7: 6, 7^2: 4, 7^3: 21\}, \{7: 5, 7^2: 13, 7^3: 6, 7^4: 7\}, \{7: 5, 7^2: 12, 7^3: 14\}, \{7: 4, 7^2: 20, 7^3: 7\}, \{7: 3, 7^2: 28\}$\\
\hline
37 & $\{7: 6, 7^2: 6, 7^3: 6, 7^4: 6, 7^5: 6, \textcolor{purple}{7^6}: 7\}, \{7: 6, 7^2: 6, 7^3: 6, 7^4: 5, 7^5: 14\}, \{7: 6, 7^2: 6, 7^3: 5, 7^4: 13, 7^5: 7\}, \{7: 6, 7^2: 6, 7^3: 4, 7^4: 21\}, \{7: 6, 7^2: 5, 7^3: 13, 7^4: 6, 7^5: 7\}, \{7: 6, 7^2: 5, 7^3: 12, 7^4: 14\}, \{7: 6, 7^2: 4, 7^3: 20, 7^4: 7\}, \{7: 6, 7^2: 3, 7^3: 28\}, \{7: 5, 7^2: 13, 7^3: 6, 7^4: 6, 7^5: 7\}, \{7: 5, 7^2: 13, 7^3: 5, 7^4: 14\}, \{7: 5, 7^2: 12, 7^3: 13, 7^4: 7\}, \{7: 5, 7^2: 11, 7^3: 21\}, \{7: 4, 7^2: 20, 7^3: 6, 7^4: 7\}, \{7: 4, 7^2: 19, 7^3: 14\}, \{7: 3, 7^2: 27, 7^3: 7\}, \{7: 2, 7^2: 35\}$\\
\end{longtable}

\small

\begin{longtable}[h]{ | p{0.06\textwidth} | >{\raggedright\arraybackslash}p{0.84\textwidth} |}
\hline
Rank &  Solutions for $p=11$  \\
\hline
11 & 
$\{\textcolor{purple}{11}:11\}$\\
\hline
21 & $\{11: 10, \textcolor{purple}{11^2}: 11\}$\\
\hline
31 &  $ \{11: 10, 11^2: 10, \textcolor{purple}{11^3}: 11\}, \{11: 9, 11^2: 22\}$ \\
\hline
41 & $\{11: 10, 11^2: 10, 11^3: 10, \textcolor{purple}{11^4}: 11\}, \{11: 10, 11^2: 9, 11^3: 22\}, \{11: 9, 11^2: 21, 11^3: 11\}, \{11: 8, 11^2: 33\}$\\
\hline
51 &  $ \{11: 10, 11^2: 10, 11^3: 10, 11^4: 10, \textcolor{purple}{11^5}: 11\}, \{11: 10, 11^2: 10, 11^3: 9, 11^4: 22\}, \{11: 10, 11^2: 9, 11^3: 21, 11^4: 11\}, \{11: 10, 11^2: 8, 11^3: 33\}, \{11: 9, 11^2: 21, 11^3: 10, 11^4: 11\}, \{11: 9, 11^2: 20, 11^3: 22\}, \{11: 8, 11^2: 32, 11^3: 11\}, \{11: 7, 11^2: 44\} $  \\
\hline
61 & $\{11: 10, 11^2: 10, 11^3: 10, 11^4: 10, 11^5: 10, \textcolor{purple}{11^6}: 11\}, \{11: 10, 11^2: 10, 11^3: 10, 11^4: 9, 11^5: 22\}, \{11: 10, 11^2: 10, 11^3: 9, 11^4: 21, 11^5: 11\}, \{11: 10, 11^2: 10, 11^3: 8, 11^4: 33\}, \{11: 10, 11^2: 9, 11^3: 21, 11^4: 10, 11^5: 11\}, \{11: 10, 11^2: 9, 11^3: 20, 11^4: 22\}, \{11: 10, 11^2: 8, 11^3: 32, 11^4: 11\}, \{11: 10, 11^2: 7, 11^3: 44\}, \{11: 9, 11^2: 21, 11^3: 10, 11^4: 10, 11^5: 11\}, \{11: 9, 11^2: 21, 11^3: 9, 11^4: 22\}, \{11: 9, 11^2: 20, 11^3: 21, 11^4: 11\}, \{11: 9, 11^2: 19, 11^3: 33\}, \{11: 8, 11^2: 32, 11^3: 10, 11^4: 11\}, \{11: 8, 11^2: 31, 11^3: 22\}, \{11: 7, 11^2: 43, 11^3: 11\}, \{11: 6, 11^2: 55\}$ \\
\hline
\caption{Small rank solutions for $p=11.$}
\label{table: }
\end{longtable}

\small

\begin{longtable}[h]{ | p{0.06\textwidth} | >{\raggedright\arraybackslash}p{0.84\textwidth} |}
\hline
Rank &  Solutions for $p=13$  \\
\hline
13 & 
$\{\textcolor{purple}{13}:13\}$\\
\hline
25 & $\{13: 12, \textcolor{purple}{13^2}: 13\}$\\
\hline
37 &  $ \{13: 12, 13^2: 12, \textcolor{purple}{13^3}: 13\}, \{13: 11, 13^2: 26\}$ \\
\hline
49 & $\{13: 12, 13^2: 12, 13^3: 12, \textcolor{purple}{13^4}: 13\}, \{13: 12, 13^2: 11, 13^3: 26\}, \{13: 11, 13^2: 25, 13^3: 13\}, \{13: 10, 13^2: 39\}$\\
\hline
61 &  $ \{13: 12, 13^2: 12, 13^3: 12, 13^4: 12, \textcolor{purple}{13^5}: 13\}, \{13: 12, 13^2: 12, 13^3: 11, 13^4: 26\}, \{13: 12, 13^2: 11, 13^3: 25, 13^4: 13\}, \{13: 12, 13^2: 10, 13^3: 39\}, \{13: 11, 13^2: 25, 13^3: 12, 13^4: 13\}, \{13: 11, 13^2: 24, 13^3: 26\}, \{13: 10, 13^2: 38, 13^3: 13\}, \{13: 9, 13^2: 52\} $  \\
\hline
73 & $\{13: 12, 13^2: 12, 13^3: 12, 13^4: 12, 13^5: 12, \textcolor{purple}{13^6}: 13\}, \{13: 12, 13^2: 12, 13^3: 12, 13^4: 11, 13^5: 26\}, \{13: 12, 13^2: 12, 13^3: 11, 13^4: 25, 13^5: 13\}, \{13: 12, 13^2: 12, 13^3: 10, 13^4: 39\}, \{13: 12, 13^2: 11, 13^3: 25, 13^4: 12, 13^5: 13\}, \{13: 12, 13^2: 11, 13^3: 24, 13^4: 26\}, \{13: 12, 13^2: 10, 13^3: 38, 13^4: 13\}, \{13: 12, 13^2: 9, 13^3: 52\}, \{13: 11, 13^2: 25, 13^3: 12, 13^4: 12, 13^5: 13\}, \{13: 11, 13^2: 25, 13^3: 11, 13^4: 26\}, \{13: 11, 13^2: 24, 13^3: 25, 13^4: 13\}, \{13: 11, 13^2: 23, 13^3: 39\}, \{13: 10, 13^2: 38, 13^3: 12, 13^4: 13\}, \{13: 10, 13^2: 37, 13^3: 26\}, \{13: 9, 13^2: 51, 13^3: 13\}, \{13: 8, 13^2: 65\}$ \\
\hline
\caption{Small rank solutions for $p=13.$}
\label{table: }
\end{longtable}

\small

\begin{longtable}[h]{ | p{0.06\textwidth} | >{\raggedright\arraybackslash}p{0.84\textwidth} |}
\hline
Rank &  Solutions for $p=17$--In progress  \\
\hline
17 & 
$\{\textcolor{purple}{17}:17\}$\\
\hline
33 & $\{17: 16, \textcolor{purple}{17^2}: 17\}$\\
\hline
49 &  $ \{17: 16, 17^2: 16, \textcolor{purple}{17^3}: 17\}, \{17: 15, 17^2: 34\}$  \\
\hline
65 & $\{17: 16, 17^2: 16, 17^3: 16, \textcolor{purple}{17^4}: 17\}, \{17: 16, 17^2: 15, 17^3: 34\}, \{17: 15, 17^2: 33, 17^3: 17\}, \{17: 14, 17^2: 51\}$ \\
\hline
81 &  $\{17: 16, 17^2: 16, 17^3: 16, 17^4: 16, \textcolor{purple}{17^5}: 17\}, \{17: 16, 17^2: 16, 17^3: 15, 17^4: 34\}, \{17: 16, 17^2: 15, 17^3: 33, 17^4: 17\}, \{17: 16, 17^2: 14, 17^3: 51\}, \{17: 15, 17^2: 33, 17^3: 16, 17^4: 17\}, \{17: 15, 17^2: 32, 17^3: 34\}, \{17: 14, 17^2: 50, 17^3: 17\}, \{17: 13, 17^2: 68\}
$\\
\hline
97 & $\{17: 16, 17^2: 16, 17^3: 16, 17^4: 16, 17^5: 16, \textcolor{purple}{17^6}: 17\}, \{17: 16, 17^2: 16, 17^3: 16, 17^4: 15, 17^5: 34\}, \{17: 16, 17^2: 16, 17^3: 15, 17^4: 33, 17^5: 17\}, \{17: 16, 17^2: 16, 17^3: 14, 17^4: 51\}, \{17: 16, 17^2: 15, 17^3: 33, 17^4: 16, 17^5: 17\}, \{17: 16, 17^2: 15, 17^3: 32, 17^4: 34\}, \{17: 16, 17^2: 14, 17^3: 50, 17^4: 17\}, \{17: 16, 17^2: 13, 17^3: 68\}, \{17: 15, 17^2: 33, 17^3: 16, 17^4: 16, 17^5: 17\}, \{17: 15, 17^2: 33, 17^3: 15, 17^4: 34\}, \{17: 15, 17^2: 32, 17^3: 33, 17^4: 17\}, \{17: 15, 17^2: 31, 17^3: 51\}, \{17: 14, 17^2: 50, 17^3: 16, 17^4: 17\}, \{17: 14, 17^2: 49, 17^3: 34\}, \{17: 13, 17^2: 67, 17^3: 17\}, \{17: 12, 17^2: 85\}$ \\
\hline
\caption{Small rank solutions for $p=17.$}
\label{table: }
\end{longtable}

%% file: Sols5.tex
\section{Solutions library: rank 5}\label{App: solutions: rank 5}

{\normalfont We list here the full list of solutions $[x_1, x_2, x_3, x_4, x_5]$ to the Egyptian fractions equation (without prime restrictions). The code used for generating these solutions is provided in Appendix \ref{appendix: one prime code}.}

{\tiny

\vspace{0.5cm}

%

}

%% file: Sols6.tex
\section{Solutions library: rank 6}\label{App: solutions: rank 6}


\tiny


%
        \end{tabular}


\newpage

\end{tabular}


\newpage

 \end{tabular}


 \newpage
\hspace{-1cm}

 \end{tabular}


\newpage

\end{tabular}


\newpage

 \end{tabular}


\newpage

 \end{tabular}


\newpage

 \end{tabular}


\newpage

 \end{tabular}


\newpage

 \end{tabular}


\newpage

 \end{tabular}


\newpage

 \end{tabular}


\newpage

 \end{tabular}


\newpage

 \end{tabular}


\newpage

 \end{tabular}


\newpage

 \end{tabular}


\newpage

 \end{tabular}


\newpage

 \end{tabular}


\newpage

 \end{tabular}


\newpage

 \end{tabular}


\newpage

 \end{tabular}


\newpage

 \end{tabular}


\newpage

 \end{tabular}


\newpage

 \end{tabular}